\documentclass{mystyle}
\usepackage{graphicx}
\usepackage{amsfonts, amssymb, amsmath, mathrsfs, latexsym}
\usepackage[all]{xy}
\usepackage[dvipsnames]{xcolor}
\usepackage{mathtools}
\usepackage{tikz}
\usetikzlibrary{positioning}
\usepackage[dvipsnames]{xcolor}
\usepackage[normalem]{ulem}
\usepackage{booktabs}

\ifpdf
	\usepackage[unicode,pdftex]{hyperref}
	\input glyphtounicode
\else
	\usepackage[unicode,dvipdfm]{hyperref}
\fi

\newtheorem{thm}{Theorem}[section]
\newtheorem{prop}[thm]{Proposition}
\newtheorem{lem}[thm]{Lemma}
\newtheorem{cor}[thm]{Corollary}

\theoremstyle{definition}

\newtheorem{deff}[thm]{Definition}
\newtheorem{defrmk}[thm]{Notations and Remark}
\newtheorem{rmk}[thm]{Remark}

\newtheorem*{mcs}{Modified Assertion of Severi}

\newcommand{\PP}{\ensuremath{\mathbb{P}}}

\newcommand{\h}{\ensuremath{\mathcal}}
\newcommand{\HL}{\ensuremath{\mathcal{H}^\mathcal{L}_}}
\newcommand{\dashdownarrow}{\mathrel{\rotatebox[origin=t]{-270}{\reflectbox{$\dashrightarrow$}}}}

\newcommand{\vni}{\vskip 4pt \noindent}

\title[Hilbert scheme of smooth curves with small index of speciality]
{On the Hilbert scheme of  linearly normal curves in $\mathbb{P}^r$ with small index of speciality}
\thanks{This work was started when the author was enjoying the hospitality and the stimulating atmosphere of the Max-Planck-Insitut f\"ur Mathematik (Bonn). The author was supported in part by National Research Foundation of South Korea (2019R1I1A1A01058457). The author would like to thank the referee for several useful comments as well as for pointing out some inaccuracies with  notations, which enhanced the overall readability of the paper significantly. }

\author[Changho Keem]{Changho Keem}
\address{
Department of Mathematics\\
Seoul National University\\
Seoul 08826,  
South Korea}
\email{ckeem1@gmail.com or ckeem@snu.ac.kr}
\date{\today}
\begin{document}
\begin{abstract}
We study the Hilbert scheme $\mathcal{H}^\mathcal{L}_{d,g,r}$ parametrizing smooth, irreducible, non-degenerate and linearly normal curves of degree $d$ and genus $g$ in $\PP^r$ whose complete and very ample hyperplane linear series $\h{D}$ have  relatively small index of speciality $i(\h{D})=g-d+r$. In particular we show the existence (and non-existence as well in some sporadic cases) of every Hilbert scheme of linearly normal curves with $i(\h{D})=4$. We also determine the irreducibility of $\HL{2r+4,r+8,r}$ for $3\le r\le 8$, which are rather peculiar families in  a certain sense.\end{abstract}
\maketitle		
\section{ An overview, preliminaries and basic set-up}

Let $\h{H}_{d,g,r}$ be the Hilbert scheme of smooth, irreducible and non-degenerate curves of degree $d$ and genus $g$ in $\PP^r$. 
We denote by $\mathcal{H}^\mathcal{L}_{d,g,r}$ the union of those components of $\mathcal{H}_{d,g,r}$ whose general element is linearly normal.
By abuse of terminology, we say that a component of the Hilbert scheme $\h{H}_{d,g,r}$ has index of speciality $\alpha$, if 
the hyperplane series $\h{D}=g^r_d$ of a general element of the component has index of speciality $\alpha$, i.e. 
$$h^1(C,\h{D})=g-d+\dim |\h{D}| =\alpha \ge g-d+r.$$
\noindent
In case $\alpha \gneq g-d+r$, the linear  series $\h{D}$ is incomplete and it is possible that there may exist a component of the Hilbert scheme which has index of speciality strictly greater than $g-d+r$; cf. \cite{FS} and references therein, especially in the paragraph before the statement of Main Theorem in the Introduction, in which the authors also provide useful remarks on  connections between Hilbert scheme business and several other areas in algebraic geometry.  However  the index of speciality of any component of $\HL{d,g,r}$ is $g-d+r$, which one may define as the {\bf index of speciality of a Hilbert scheme of linearly normal curves}. 

\vni
\noindent
We recall the following modified assertion of Severi, which has been given attention by some authors recently; cf. \cite{JPAA, KK3}.
\begin{mcs}  A nonempty $\mathcal{H}^\mathcal{L}_{d,g,r}$ is irreducible for any triple $(d,g,r)$ in the Brill-Noether range  $$\rho (d,g,r)=g-(r+1)(g-d+r)\ge 0,$$ where $\mathcal{H}^\mathcal{L}_{d,g,r}$ is the union of those components of ${\mathcal{H}}_{d,g,r}$ whose general element is linearly normal.
\end{mcs}

\noindent
We note that the modified assertion of Severi makes sense only if the index of speciality of Hilbert scheme of linearly normal curves $\HL{d,g,r}$ is non-negative.  Through a preliminary attempt to settle down the Modified Assertion of Severi, which is still in its infancy, one has  quite extensive knowledge about the Hilbert scheme of linearly normal curves of small index of speciality $\alpha \le 3$ which can be summarized as follows; cf. \cite{lengthy, JPAA, KK3}.

\begin{itemize}
\item{$\alpha =0$}: $\mathcal{H}^\mathcal{L}_{g+r,g,r}\neq\emptyset$ and irreducible.
\item{ $\alpha=1$}: $\mathcal{H}^\mathcal{L}_{g+r-1,g,r}\neq\emptyset$ and irreducible for $g\ge r+1$ and is empty for $g\le r$.
\item{$\alpha=2$}: $\mathcal{H}^\mathcal{L}_{g+r-2,g,r}\neq\emptyset$ and irreducible for every $g\ge r+3$ and is empty if 
 $g\le r+2$.

\item{$\alpha=3$}: $\h{H}^\h{L}_{g+r-3,g,r}=\emptyset$ for $g\le r+4$.
For $r\ge 5$ and $g\ge r+5$, $\h{H}^\h{L}_{g+r-3,g,r}\neq\emptyset$ unless $g=r+6$ and $r\ge 10$. Moreover $\h{H}^\h{L}_{g+r-3,g,r}$ is irreducible for every $g\ge 2r+3$ and is reducible for almost all $g$ in the range $r+5\le g\le 2r+2$.

\end{itemize}

\noindent
\vni
One obvious advantage in considering the Hilbert scheme of linearly normal curves according to its  index of speciality {\bf $\alpha$} is that  the residual series of the very ample hyperplane series corresponding to a general element of any component  of the Hilbert scheme has fixed dimension $\alpha -1$ regardless of values of  genus $g$ and  
degree $d$. Therefore one may work  more effectively in exploring out several properties of  Hilbert schemes under consideration by looking at the family of curves in a fixed projective space $\PP^{\alpha -1}$ induced by the residual series of  hyperplane series. Especially when $\alpha$ is low, approaching this way becomes somewhat useful.

\vni
\noindent
In case $\alpha =3$, 
the residual series of hyperplane series are nets and it is possible to derive the basic properties such as irreducibility,  existence as well as the number of moduli of the Hilbert scheme of linearly normal curves by considering the corresponding property of the Severi variety of plane curves,  which has been proven to be effective to a certain extent in some previous works on the subject; \cite{lengthy, JPAA, KKy1, KK3}. In this paper we consider Hilbert scheme of linear normal curves with index of speciality $\alpha =4$. Equivalently we study the Hilbert schemes of non-degenerate, smooth and linearly normal curves in $\PP^r$ ($r\ge 3$)  of degree $d=g+r-4$. We show the existence as well as non-existence of $\HL{d,g,r}$ for every triple
$(d,g,r)$ with $g-d+r=4$. We also determine the irreducibility of $\HL{d,g,r}$ for specific triples $(d,g,r)=(2r+4,r+8,r)$ with $3\le r\le 8$ which is of some particular interest. We also study $\HL{d,g,r}$ for the triples $(d,g,r)=(2r+5,r+9,r)$ ($3\le r\le 11$) and give a partial irreducibility result. 

\vni
\noindent
For notation and conventions, we  usually follow those in \cite{ACGH, ACGH2, H1}; e.g. $\pi (d,r)$ is the maximal possible arithmetic genus of an irreducible and non-degenerate curve of degree $d$ in $\PP^r$. $\pi_1(d,r)$ is the so-called the second Castelnuovo genus bound of  an irreducible and non-degenerate curve of degree $d$ in $\PP^r$  not lying on a surface of degree $r-1$; cf. \cite[Theorem 3.13, 3.15]{H1}. In particular if $r=3$, $\pi_1(d,3)$ is the maximal possible arithmetic genus of an irreducible and non-degenerate curve of degree $d$ in $\PP^3$ not lying on a quadric surface. 
Following classical terminology, a linear series of degree $d$ and dimension $r$ on a smooth irreducible curve $C$ is usually denoted by $g^r_d$.
A base-point-free linear series $g^r_d$ ($r\ge 2$) on a smooth curve $C$ is called birationally very ample when the morphism 
$C \rightarrow \mathbb{P}^r$ induced by  $g^r_d$ is generically one-to-one (or birational) onto its image.  A base-point-free linear series $g^r_d$ on $C$  is said to be compounded of a covering (compounded for short) if the morphism induced by $g^r_d$ gives rise to a non-trivial covering map $C\rightarrow C'$ of degree $k\ge 2$. For a complete linear series $\h{E}$ on a smooth curve $C$, the residual series $|K_C-\h{E}|$ is denoted by $\h{E}^\vee$.
We work over the field of complex numbers.

\vni
\noindent
The organization of this paper is as follows. We briefly recall some terminologies and basic preliminaries  in the remainder of this section. In the second section we determine all the triples $(d,g,r)$,  $r\ge 3$ for which $\mathcal{H}^\mathcal{L}_{g+r-4,g,r}\neq\emptyset$.

\vni
\noindent
In subsequent sections, we determine the irreducibility of $\mathcal{H}^\mathcal{L}_{g+r-4,g,r}$ for $g=r+8$ by using the  irreducibility of Severi varieties of nodal curves on a non-singular quadric surface in $\PP^3$. In the final section -- with an open end -- we make a discussion on a couple of related aspects of our study, such as the next case $d=g+r-5$ for further investigation, in addition to the existence and partial irreducibility result of  $\mathcal{H}^\mathcal{L}_{g+r-4,g,r}$ for $g=r+9$  ($3\le r\le 11$). 

\vni
\noindent
We briefly recall several fundamental facts for our study of Hilbert schemes  which are  well-known; cf. \cite{ACGH2}  or \cite[\S 1 and \S 2]{AC2}.
Let $\mathcal{M}_g$ be the moduli space of smooth curves of genus $g$. Given an isomorphism class $[C] \in \mathcal{M}_g$ corresponding to a smooth irreducible curve $C$, there exist a neighborhood $U\subset \mathcal{M}_g$ of the class $[C]$ and a smooth connected variety $\mathcal{M}$ which is a finite ramified covering $h:\mathcal{M} \to U$, as well as  varieties $\mathcal{C}$, $\mathcal{W}^r_d$ and $\mathcal{G}^r_d$ proper over $\mathcal{M}$ with the following properties:
\begin{enumerate}
\item[(1)] $\xi:\mathcal{C}\to\mathcal{M}$ is a universal curve, i.e. for every $p\in \mathcal{M}$, $\xi^{-1}(p)$ is a smooth curve of genus $g$ whose isomorphism class is $h(p)$,
\item[(2)] $\mathcal{W}^r_d$ parametrizes the pairs $(p,L)$ where $L$ is a line bundle of degree $d$ and $h^0(L) \ge r+1$ on $\xi^{-1}(p)$,
\item[(3)] $\mathcal{G}^r_d$ parametrizes the couples $(p, \mathcal{D})$, where $\mathcal{D}$ is possibly an incomplete linear series of degree $d$ and dimension $r$ on $\xi^{-1}(p)$.
\end{enumerate}

\vni
\noindent
Given an irreducible family $\h{F}\subset\h{G}^r_d$ with some geometric meaning, if a general element of $\h{F}$ is complete,  the closure of the family $\{\h{E}^\vee| \h{E}\in \h{F}, ~\h{E} \textrm{ is complete}\}\subset \h{G}^{g-d+r-1}_{2g-2-d}$ is sometimes denoted by $\h{F}^\vee$.

\vni
\noindent
Let $\widetilde{\mathcal{G}}$ ($\widetilde{\mathcal{G}}_\mathcal{L}$,  respectively) be  the union of components of $\mathcal{G}^{r}_{d}$ whose general element $(p,\mathcal{D})$ corresponds to a very ample (very ample and complete, respectively) linear series $\mathcal{D}$ on the curve $C=\xi^{-1}(p)$. By recalling that an open subset of $\mathcal{H}_{d,g,r}$ consisting of elements corresponding to smooth irreducible and non-degenerate curves is a $\mathbb{P}\textrm{GL}(r+1)$-bundle over an open subset of $\widetilde{\mathcal{G}}$, the irreducibility of $\widetilde{\mathcal{G}}$ guarantees the irreducibility of $\mathcal{H}_{d,g,r}$. Likewise, the irreducibility of $\widetilde{\mathcal{G}}_\mathcal{L}$ ensures the irreducibility of 
 $\mathcal{H}_{d,g,r}^\mathcal{L}$.
We recall the following regarding the schemes $\mathcal{G}^{r}_{d}$ which is also well-known; cf. \cite[2.a]{H1} and \cite[Ch. 21, \S 3, 5, 6, 11, 12]{ACGH2}. 
\begin{prop}\label{facts}
For non-negative integers $d$, $g$ and $r\ge 1$, let $$\rho(d,g,r):=g-(r+1)(g-d+r)$$ be the Brill-Noether number.
The dimension of any component of $\mathcal{G}^{r}_{d}$ is at least $$\lambda(d,g,r):=3g-3+\rho(d,g,r), $$ 
hence the minimal possible dimension of any component of $\h{H}_{d,g,r}$ is
$$\h{X}(d,g,r):=\lambda (d,g,r)+\dim\PP GL(r+1).$$ Moreover, if $\rho(d,g,r)\ge 0$, there exists a unique component (called the principal component) $\mathcal{G}_0$ of $\widetilde{\mathcal{G}}$ which dominates $\mathcal{M}$(or $\mathcal{M}_g$).
	\end{prop}
\section{Existence and non-existence of $\HL{d,g,r}$}
In this section we prove the existence of $\HL{g+r-4,g,r}$ for every $g\ge r+7$ ($r\ge 5$) unless  $g=r+8$ ($r\ge 9$) or $g=r+9$  ( $r\ge 12)$, in which cases we have $\HL{g+r-4,g,r}=\emptyset$.  
For  $r=3,~ 4$, we quote the existence of 
$\HL{g+r-4,g,r}$ which are known to some people and to the author  (or as folklores)  as follows. 

\begin{rmk}\label{r=34}
(i) For $r=3,$ it is easy to see that  $$\HL{g+r-4,g,3}=\HL{g-1,g,3}\neq\emptyset \textrm{ if and only if }g\ge r+6=9.$$ Note that  by Castelnuovo genus bound, $\h{H}^\h{L}_{g-1,g,3}=\emptyset$ for $g\le 8$.  For $g=9$, (extremal) curves  of type $(4,4)$ on  quadrics form the only irreducible component of degree $d=8$ and genus $g=9$ of dimension 
$$\dim\PP H^0(\PP^1\times\PP^1,\h{O}_{\PP^1\times\PP^1}(4,4))+\dim\PP H^0(\PP^3,\h{O}_{\PP^3}(2))=33\gneq 4\cdot 8=\h{X}(8,g,3).$$
For $g=10$, there exist curves of type $(3,6)$ on  smooth quadrics which form an irreducible family of dimension 
$$\dim\PP H^0(\PP^1\times\PP^1, \h{O}_{\PP^1\times\PP^1}(3,6))+\dim\PP H^0(\PP^3,
\h{O}_{\PP^3}(2))=36.$$ Moreover, complete intersections of two irreducible cubics form another family of the same expected dimension 
$$\dim\mathbb{G}(1, \PP(H^0(\PP^3,\h{O}_{\PP^3}(3)))=\dim\mathbb{G}(1,19)=36.$$ Hence $\h{H}_{9,10,3}=\h{H}^\h{L}_{9,10,3}\neq\emptyset$ and is reducible.  We further note that the inequality $$g\le\frac{(g-2)(g-3)}{6}=\frac{(d-1)(d-2)}{6}\le\pi_1(d,3)$$ is valid for every pair $(d,g)=(g-1,g)$ with $g\ge 11$,  which guarantees the existence of a smooth curve of degree $g-1$ and genus $g$ in $\PP^3$; cf. \cite[Theorem 2.4]{Gruson}. 

\vni
\noindent
(ii) For the existence of smooth curves of genus $g$ and degree $d=g+r-4=g$ in $\PP^4$, one may deduce that 
$$\h{H}^\h{L}_{g+r-4,g,r}=\HL{g,g,4}\neq\emptyset \textrm{ if and only if } g\ge 11=r+7 $$   as follows. In general, in the range $d\le g+r$ inside the Brill-Noether range $\rho (d,g,r)\ge 0$, the principal component $\mathcal{G}_0$ which has the expected dimension $\lambda (d,g,r)$ is one of the components of $\widetilde{\mathcal{G}}_\mathcal{L}$ (cf. \cite[2.1 page 70]{H1}),   and  hence $\widetilde{\mathcal{G}}_\mathcal{L}$ or 
 $\mathcal{H}_{d,g,r}^\mathcal{L}$ is non-empty. In particular for $g\ge 20$, $\HL{g,g,4}\neq\emptyset$.
 For $(d,g,4)=(g,g,4)$ outside the Brill-Noether range $11\le g\le 19$, we refer \cite[Theorem 1.0.2]{rathman}. We also remark that the proof of our existence theorem (Theorem \ref{main}) works for several cases $(d,g,r)=(g,g,4)$ with negative Brill-Noether number.  For $g=16$ or $g=19$, there exist smooth curves of degree $d=g$  on Bordiga surface in $\PP^4$. The Castelnuovo genus bound implies the the non-existence of $\HL{g,g,4}$ for $g\le 10$; cf. Table 1 at the end of this section.
\end{rmk}

\vni
\noindent
The following lemma - regarding multiples of the unique pencil $g^1_k$ on a general $k$-gonal curve - is useful  to show the existence of a very ample linear series having index of speciality $\alpha \le 4$. 
\begin{lem}\cite[Proposition 1.1]{CKM}\label{kveryample} Assume $2k-g-2<0$. Let $C$ be a general $k$-gonal curve of genus $g$, $k\ge 2$, $0\le m$, $n\in\mathbb{Z}$ such that 
\begin{equation}\label{veryamplek}
g\ge 2m+n(k-1)
\end{equation}
 and let $D\in C_m$ be an effective divisor of degree $m$ on $C$. Assume that there is no $F\in g^1_k$ with $F\le D.$ Then $\dim|ng^1_k+D|=n$.
\end{lem}

\vni
\noindent
The  following easy lemma -- which is adapted to our specific situation -- will be used to show the 
non-existence of $\HL{g+r-4,g,r}$ for $g=r+8$ ($r\ge 9$) and $g=r+9$ ($ r\ge12$). 

\begin{lem}\label{easylemma} Let $\h{E}=g^3_e$ (with $e=10, 11$) be a complete and special linear series (possibly with non-empty base locus $\Delta$)  on a smooth curve $C$ of genus $g$. Suppose $\h{E}$ is compounded. Then $\h{E}^\vee$ -- the residual series of $\h{E}$ -- is not very ample. 
\end{lem}
\begin{proof} Let $k$ be the degree of the morphism $C\stackrel{\phi}{\rightarrow} E\subset\PP^3$ induced by 
the base-point-free part of $\h{E}$. Since $e\ge k\cdot\deg\phi (C)$ and $\deg\phi (C)\ge3$, the following three cases are possible; 
\begin{enumerate}
\item[(a)] $k=3$ and  $C$ is 
 trigonal with $\h{E}=|3g^1_3+\Delta|$, $\Delta\neq\emptyset$.
 \item[(b)] $k=2$ and $C$ is bielliptic with $\h{E}=|\phi^*(g^3_4)+\Delta|$, where $C\stackrel{\phi}{\rightarrow} E$ is a double cover of an elliptic curve $E$. 
 \item[(c)] $k=2$ and $C \stackrel{\phi}{\rightarrow} E$ is a double cover of a curve $E$ of genus $h=2$, $\h{E}=|\phi^*(g^3_5)+\Delta|$. 
 \end{enumerate}
(a) If $C$ is trigonal, let $q\in\Delta$ and  let  $q+t+s\in g^1_3$ be the unique trigonal divisor containing  $q$. We then have 
\begin{eqnarray*}
\dim|\h{E}+t+s|&=&\dim|3g^1_3+\Delta+t+s|=\dim|3g^1_3+(q+t+s)+(\Delta-q)|\\&=&\dim|4g^1_3+(\Delta -q)|\ge\dim\h{E}+1.
\end{eqnarray*}
 \noindent
 (b) If
 $k=2$ and  $C$ is bielliptic, then for any $u\in E$, 
 $$\dim|\h{E}+\phi^*(u)|=\dim|\phi^*(g^3_4+u)+\Delta|=\dim|\phi^*(g^4_5)+\Delta|\ge \dim\h{E}+1$$
 
 \noindent
 (c) For any $u\in E$, we also have $$\dim|\h{E}+\phi^*(u)|=\dim|\phi^*(g^3_5+u)+\Delta|=\dim|\phi^*(g^4_6)+\Delta|\ge \dim\h{E}+1.$$ 
 \noindent
 By Riemann-Roch,  it follows that $$\dim|\h{E}^\vee -(r_1+r_2)|\ge \dim\h{E}^\vee-1$$ for some $r_1+r_2\in C_2$ and hence $\h{E}^\vee$ is not very ample.
 \end{proof}

\noindent
We assume $r\ge 5$ in the next theorem, just because the existence of   $\HL{g+r-4,g,r}$ for $r=3, 4$ is completely known as we discussed  in Remark \ref{r=34}.
\begin{thm} \label{main}
\begin{enumerate}
\item[(a)] $\HL{g+r-4,g,r}=\emptyset$ for $g\le r+6$, $r\ge 5$. 
\item[(b)] $\HL{g+r-4,g,r}\neq\emptyset$
for  $g=r+7$ or for any $g\ge r+10$, $r\ge 5$.
\item[(c)] $\HL{g+r-4,g,r}\neq\emptyset$ if and only if $~5\le r\le 8$ for $g=r+8$, $r\ge 5$ .
\item[(d)] $\HL{g+r-4,g,r}\neq\emptyset$ if and only if $~5\le r\le 11$ for $g=r+9$, $r\ge 5$. 
\end{enumerate}
\end{thm}
\begin{proof} (a) For $r\ge 5$, we have $g\le \pi (g+r-4,r)=2g-r-7$ if $m:=[\frac{g+r-5}{r-1}]=2$ (or $g\le \pi (g+r-4,r)=g-4$ if $m=1$) by Castelnuovo genus bound, which is not compatible with the assumption $g\le r+6$. 

\vni
\noindent
(b)  For $g=r+7$ and $r\ge 5$, we have $\pi (g+r-4,r)=r+7$. Hence the curve corresponding to a  general  element of a component of $\HL{2r+3,r+7,r}$ is an extremal curve, a curve whose arithmetic genus attains its maximal value $\pi(d,r)$. Therefore $\HL{2r+3,r+7,r}\neq\emptyset$ which is irreducible except for $r=5$; cf. \cite[Corollary 3.12, page 92]{H1}.
We now show the existence of $\HL{g+r-4,g,r}$ for every $g\ge r+10$ and $r\ge 5$. We will establish the existence of corresponding curves in $\HL{g+r-4,g,r}$ on a suitable family of $k$-gonal curves. 

\vni
Let $e:=g-r+2=\deg\h{E}=\deg |K_C-\h{D}|$;  $(p,\h{D})\in\h{G}\subset\widetilde{\h{G}}_\h{L}\subset\h{G}^r_{g+r-4}$, $C=\xi^{-1}(p)$. We distinguish the following three cases.
\begin{enumerate}
\item[(i)] $e\equiv 0 ~(\textrm{mod}~3)$: By the assumption $g\ge r+10$, we may set $e=g-r+2=3k\ge 12$ for some $k\ge 4$.  Note that $$\rho (k,g,1)=2k-g-2=-k-r\lneqq 0$$ and we consider a general $k$-gonal curve $C$.
Taking  $n=3, m=2$  in Lemma \ref{kveryample}, the numerical condition (\ref{veryamplek}) is satisfied; note that $g=3k+r-2\ge 3k+1$ for  $r\ge 3$. Furthermore there is no $F\in g^1_k$ with $F\le D$ (where $D\in C_2$) since $k\ge 4$. Hence we have $$\dim|3g^1_k+D|=3$$ for any $D\in C_2$ by Lemma \ref{kveryample} and therefore $$|K_C-3g^1_k|=g^r_{g+r-4}$$ is very ample.

\noindent
\item[(ii)] $e\equiv 1 ~(\textrm{mod}~3)$: Again, by the assumption $g\ge r+10$, we may set $e=g-r+2=3k+1\ge 12$ for some $k\ge 4$.  Note that $$\rho (k,g,1)=2k-g-2=-k-r-1\lneqq 0$$ and we again consider a general $k$-gonal curve $C$.
 For a general choice
$q\in C$ and for any $t+s\in C_2$, we take $D=q+t+s$, $m=3$ and $n=3$ in Lemma \ref{kveryample}. The numerical condition  (\ref{veryamplek}) holds ($g=3k+r-1\ge 3k+3$ if $r\ge 4$) and there is no $F\in g^1_k$ with $F\le D$ just because $k\ge 4$. Hence $\dim |3g^1_k+D|=3$ which implies that $$|K_C-3g^1_k-q|=g^{r}_{2g-2-3k-1}=g^r_{g+r-4}$$ is very ample.

\noindent
\item[(iii)] $e\equiv 2 ~(\textrm{mod}~3)$: Again, by the assumption $g\ge r+10$, we may set $e=g-r+2=3k+2\ge 12$ for some $k\ge 4$.  Note that $$\rho (k,g,1)=2k-g-2=-k-r-2\lneqq 0$$ and we consider a general $k$-gonal curve $C$.
 We take $q+q'\in C$ such that $q+q'$ is not in the same fiber of the $k$-sheeted map onto $\PP^1$ defined by the unique $g^1_k$.
For any $t+s\in C_2$, we take $D=q+q'+t+s$, $m=4$ and $n=3$ in Lemma \ref{kveryample}. Again the numerical condition  (\ref{veryamplek}) holds (by $r\ge 5$) and there is no $F\in g^1_k$ with $F\le D$ by our choice of $q+q'\in C_2$. Hence $\dim |3g^1_k+D|=3$ which implies that $$|K_C-3g^1_k-q-q'|=g^{r}_{2g-2-3k-2}=g^r_{g+r-4}$$ is very ample.

\noindent
We remark that if $r=3$ and  $g-1$ is a multiple of $3$, the proof (i) in the above is still valid and $$|K_C-3g^1_k|=g^3_{g-1}$$  is very ample.  Hence $\HL{g-1,g,3}\neq\emptyset$ for $g$ such that $g-1\equiv 0 ~(\textrm{mod}~3)$.
Likewise, for $r=4$ and if $g-2$ is a multiple of $3$, then $$|K_C-3g^1_k|=g^4_g$$ is very ample. Moreover, if $g-2=3k+1$ for some $k\ge 4$ and $r=4$, the proof (ii) in the above remains valid and one may deduce that
$$|K_C-3g^1_k-q|=g^4_{g}$$ is very ample for a general $q\in C$. From this observation we  have the 
existence of a smooth linearly normal curves of degree $d$ and genus $g$ in $\PP^4$ for $(d,g)=(11,11), (14,14), (15,15), (17,17), (18,18)$; cf. Table 1.
\end{enumerate}

\vni
(c) We show the non-existence of $\h{H}^\h{L}_{g+r-4,g,r}$ in case $g=r+8$ and  $r\ge 9$ ($g=r+9$ and $r\ge 12$, respectively) as follows. We set $\h{E}:=g^3_{g-r+2}=g^3_{e}=\h{D}^\vee$ for a general very ample $(p,\h{D})\in\h{G}\subset\widetilde{\h{G}}_\h{L}\subset\h{G}^r_{g+r-4}$ which has degree $e=10$ ($e=11$, respectively) for $g=r+8$ (for $g=r+9$, respectively). If $g=r+8 \gneqq \pi (10,3)=16$ (if $g=r+9 \gneqq \pi (11,3)=20$, respectively), $\h{E}$ is compounded inducing 
a $k$-sheeted map onto a curve of degree $f\ge 3$ in $\PP^3$. Since $k\cdot f\le e$, a simple numerical calculation leads to the following  possibilities; 

$(k,f)=(2,5)$ and $C$ is a double cover of a curve of genus $h=2$

$(k,f)=(2,4)$ and $C$ is bielliptic

$(k,f)=(2,3)$ and $C$ is hyperelliptic

$(k,f)=(3,3)$ and $C$ is trigonal and $\h{E}$ has non-empty base locus.

\noindent
However $\h{D}=\h{E}^\vee$ is not very ample by Lemma \ref{easylemma}, leading to a contradiction. 

\vni
\noindent
For the existence of $\h{H}^\h{L}_{g+r-4,g,r}=\h{H}^\h{L}_{2r+4,r+8,r}$ with $5\le r\le 8$ and $g=r+8$, we make the following usual convention. Let $S_{t+1}\stackrel{\pi}{\rightarrow}\PP^2$ be the surface blown up $\PP^2$ at $t+1$ general points $\{p_1,\cdots ,p_{t+1}\}\subset\PP^2$ with exceptional divisors $E_1,\cdots , E_{t+1}$ whose linear equivalence classes are denoted by $\{e_i\}_{i=1}^{t+1}$. Let $l\in \text {Pic}(S_{t+1})$ be the class of $\pi^*(L)$, where $L\subset\PP^2$ is a line.
We use the abbreviation 
$$(a;b_1,\cdots ,b_{t+1})$$ for the linear system 
$$|\pi^*(aL)-\sum_{i=1}^{t+1} b_iE_i| $$ on $S_{t+1}$.  

\vni
\noindent
Smooth curves in the very ample   linear system $(8;3^2,2^2):=(8;3,3,2,2)$ on $S_4$
embedded by the 
anti-canonical linear system $(3;1^4):=(3;1,1,1,1)$ has degree 
$$(8;3^2,2^2)\cdot (3;1^4)=14=2\cdot 5+4=2r+4$$ in $\PP^5$ and   genus 
$$g=p_a(8l)-2\cdot 3-2\cdot 1=13=5+8=r+8.$$ Likewise, a smooth curve in $(8;3,3,2)$ ($(8;3,3)$ resp.) is embedded
into $\PP^6$ ($\PP^7$ resp.)  as a curve of degree $d=16$ ($18$ resp.) and genus $g=14$ ($g=15$ resp.).
For $r=8$, a curve of degree $e=10$ and genus $g=16$ in $\PP^3$ - which is a (smooth) extremal curve of type $(5,5)$ on a (smooth) quadric surface $X$ - is embedded into $\PP^8$ by the $2$-tuple embedding
$$\PP^1\times\PP^1\stackrel{(2,2)}{\hookrightarrow}\PP^2\times\PP^2\subset\PP^8$$ induced by the linear system
$|\h{O}_{\PP^1\times\PP^1}(2,2)|$
on $\PP^1\times\PP^1$  as a curve of degree $$d=(5,5)\cdot (2,2) =2\cdot e=20=2\cdot  8+4.$$ We also see that   $\h{H}^\h{L}_{2r+4,r+8,r}=\h{H}^\h{L}_{20,16,8}$ is irreducible, which can be argued as follows. We recall that a  general $\h{E}\in\h{G}^\vee\subset\h{G}^3_e=\h{G}^3_{10}$
($\h{G}$ is a component of $\widetilde{\h{G}}_\h{L}$)  induces a morphism onto a smooth curve of type $(5,5)$ on a smooth quadric in $\PP^3$. On the other hand,  curves of type $(5,5)$ on smooth quadrics in $\PP^3$ form the {\bf only} irreducible component of  $\HL{10,16,3}$.  Therefore the family of curves induced by the residual series of the (very ample) hyperplane series of general members of the irreducible Hilbert scheme $\HL{10,16,3}$ is also irreducible, which is our $\HL{20,16,8}$. 
 $\h{H}^\h{L}_{20,16,8}$ has  dimension more than expected;
\begin{eqnarray*}
\dim\h{G}&=&\dim\h{G}^\vee=\dim\PP H^0(\PP^1\times\PP^1, \h{O}_{\PP^1\times\PP^1}(5,5))\\&+&\dim\PP H^0(\PP^3,\h{O}_{\PP^3}(2))-\dim\PP GL(4)>\lambda (d,g,r).
\end{eqnarray*}

\noindent
We will see in the next section that curves of this type -- curves embedded into $\PP^r$ by the residual series of the hyperplane series of appropriate (nodal) curves on a smooth quadric in $\PP^3$ -- form the only irreducible component of  $\HL{2r+4,r+8,r}$ for $5\le r\le 8$ except  $r=7$. For $r=7$, there exists another irreducible component consisting of  images of  smooth curves in $|\h{O}_{\PP^1\times\PP^1}(4,6)|$ on smooth quadrics via the embedding 
$$\PP^1\times\PP^1\stackrel{(1,3)}{\hookrightarrow}\PP^1\times\PP^3\subset\PP^7$$
which have degree $(4,6)\cdot (1,3)=18=2\cdot 7+4$. The very ampleness of linear systems $(a; b_1,\cdots ,b_{t+1})$ on $S_{t+1}$ appearing in this proof follows from \cite[Theorem]{sandra}.

\vni
(d) This will be treated in the final section together with some partial irreducibility result; cf.  Table \ref{3r11}.
\end{proof}

\noindent
The following table exhibits the existence of smooth curves of degree $d=g$ in $\PP^4$, which is a partial application of the proof of Theorem \ref{main} (b). We stress that curves described in the list may not form an open dense subset of a component. Furthermore the irreducibility of $\HL{g,g,4}$ is not known (to the author) in general except for some particular cases. A smooth curve of degree $d=11$ and genus $g=11$ in $\PP^4$ is not extremal but nearly extremal ($\pi_1(d,4)<g<\pi (11,4)$), which  lies on a rational normal scroll in $\PP^4$ and such curves form one irreducible family; cf.  \cite[Corollary 3.16, p. 100]{H1}. The irreducibility of $\HL{d,g,4}$ for  $(d,g)=(12,12)$ needs more argument,  which will be shown in the section following the next section. In the table, $X$ denotes a smooth quadric surface in $\PP^3$.

\begin{table}[ht]
\caption{ Smooth curves in $\h{H}^\h{L}_{g,g,4}$ for $11\le g\le 19$} 
\centering 
\begin{tabular}{c c c c c c} 
\hline\hline 
{\footnotesize $(d,g)$} & {\footnotesize Description }  & {\footnotesize Remark} \\ [0.5ex] 
\hline 
\hline 
{\footnotesize $(11,11)$} & {\footnotesize A curve in $|3H+2L|$ on a rational normal scroll in $\PP^4$ }  & {\footnotesize Trigonal} \\
{\footnotesize } & {\footnotesize with a very ample $|K_C-3g^1_3|$}  & {\footnotesize}\\
{\footnotesize $(12,12)$} & {\footnotesize A smooth member in $(8; 3^2,2^3):=(8; 3,3,2,2,2)$} & {\footnotesize Pentagonal}  \\ 
{\footnotesize } & {\footnotesize  on a Del Pezzo, and  is birational to a curve in $|(5,5)|$} & {\footnotesize }  \\
{\footnotesize } & {\footnotesize on $X\subset\PP^3$ with $4$ nodes} & {\footnotesize }  \\
{\footnotesize $(13,13)$} & {\footnotesize A smooth member in $(8; 3,3,2,2,1)$ on a Del Pezzo, } & {\footnotesize Pentagonal}  \\ 
{\footnotesize } & {\footnotesize birational to a curve in $|(5,5)|$ on $X\subset\PP^3$ with $3$ nodes}  & {\footnotesize }  \\
{\footnotesize $(14,14)$} & {\footnotesize  A curve with a very ample $|K_C-3g^1_4|$  }& {\footnotesize Tetragonal} \\ 
{\footnotesize $(15,15)$} & {\footnotesize A curve with a very ample $|K_C-3g^1_4-q|$  }& {\footnotesize Tetragonal} \\ 
{\footnotesize $(16,16)$} & {\footnotesize On a Bordiga surface}  & {\tiny \cite[Theorem 1.0.2]{rathman} }  \\
{\footnotesize $(17,17)$} & {\footnotesize A curve with a very ample $|K_C-3g^1_5|=g^4_{17}$}  & {\footnotesize Pentagonal}  \\
{\footnotesize $(18,18)$} & {\footnotesize A curve with a very ample $|K_C-3g^1_5-q|=g^4_{18}$}  & {\footnotesize Pentagonal}  \\
{\footnotesize $(19,19)$}& {\footnotesize On a Bordiga surface}  & {\tiny \cite[Theorem 1.0.2]{rathman} }  \\ [1ex] 
\hline 
\end{tabular}
\label{gg4} 
\end{table}
\section{Irreducibility of $\HL{g+r-4,g,r}$ for $g=r+8$ }
For $g=r+8$, $\HL{g+r-4,g,r}$ is non-empty only for $3\le r\le 8$ by Remark \ref{r=34} and Theorem \ref{main}. In the next two sections, we determine the irreducibility of $\HL{2r+4,r+8,r}$.
For this, we need to recall a couple of generalities regarding the Severi variety of nodal curves on a smooth quadric surface $X\subset\PP^3$.

\begin{deff}
\begin{enumerate}
\item[(i)]
Let $X=\PP^1\times\PP^1\subset\PP^3$ be a smooth quadric surface and we fix $\h{M}=|\h{O}_{\PP^1\times\PP^1}(a,b)|$ a very ample linear system on $X$. 
Let $$p_a(\h{M})=(a-1)(b-1)$$ be the arithmetic genus of curves belonging to $\h{M}$. 

\item[(ii)]
Given an integer $g$ such that $0\le g\le p_a(\h{M})$, we set
 $\delta =p_a(\h{M})-g$. We denote by $$\Sigma_{\h{M}, \delta}\subset |\h{M}|=\PP(H^0(X,\h{M}))$$ the (equi-singular) Severi variety which is the closure of the locus of integral curves in the linear system $\h{M}$ whose singular locus consists of exactly $\delta$ nodes. 
\item[(iii)]
Let $\Sigma_{\h{M},g}$ be the  (equi-generic) Severi variety which is the closure of the locus of integral curves of geometric genus $g$  in the linear system $\h{M}$. 
\end{enumerate}
\end{deff}

\noindent
We shall make use the following well-known results adopted for our current specific situation. The results quoted below are known to be true in a more general context. Readers are advised to refer \cite{DS}, which provides an excellent account on Severi varieties on surfaces in general. 

\begin{rmk}\label{Severi}
\begin{enumerate}
\item[(i)] A general member of every irreducible component of the equi-generic Severi variety $\Sigma_{\h{M},g}$ is a nodal curve; cf. \cite[Proposition 2.1]{H2}.
\item[(ii)] The equi-singular Severi variety $\Sigma_{\h{M}, \delta}$ is  irreducible of dimension $$\dim|\h{M}|-\delta,$$
if non-empty; cf. \cite[Proposition 2.11, Theorem 3.1]{tyomkin}.
\end{enumerate}
\end{rmk}
\begin{defrmk}
\begin{enumerate}\item[(i)] Let $\h{G}_{\Sigma_{\h{M},g}}\subset \h{G}^3_{e}$ be  the locus which the  irreducible equi-generic Severi variety $\Sigma_{\h{M},g}$ sits over. To be precise, let $\h{G}_{\Sigma_{\h{M},g}}\subset \h{G}^3_{e}$ be the locus consisting of (complete) webs cut out by hyperplanes on space curves corresponding to  general elements of  the Severi variety ${\Sigma_{\h{M},g}}$. We note that an open subset of $\Sigma_{\h{M},g}$ consisting of nodal curves is an $\textrm{Aut}(\PP^1\times\PP^1)$-bundle over an open subset of $\h{G}_{\Sigma_{\h{M},g}}=\h{G}_{\Sigma_{\h{M},\delta}}$.  Since $\Sigma_{\h{M},g}$ is irreducible,  $\h{G}_{\Sigma_{\h{M},g}}$ is  irreducible.
\item[(ii)] For each component $\h{G}\subset\widetilde{\h{G}}_\h{L}\subset\h{G}^r_d$, we set 
$$\h{G}^\vee :=\{\h{D}^\vee|\h{D}\in\h{G}\}\subset\h{G}^3_e$$
where $\h{D}^\vee$ is the residual series of $\h{D}\in\h{G}$, $e=2g-2-d$ and $d=g+r-4$.
We also set 
$$\h{G}':=\underset{\h{G}\subset\widetilde{\h{G}}_\h{L}}{\bigcup} \h{G}^\vee\subset\h{G}^3_e$$

\item[(iii)]
We will determine the irreducibility of $\HL{d,g,r}$ ($d=g+r-4, ~g=r+8, ~5\le r\le 8$)  according to the following intermediate steps.  
\vni
(a)
We first check that for a general $\h{D}\in\h{G}$,
$\h{E}=\h{D}^\vee$ is birationally very ample and induces a morphism into a quadric surface in $\PP^3$. 

\vni
(b) We then collect all the possible $\h{M}_j=|\h{O}_{\PP^1\times\PP^1}(c_j,d_j)|\in Pic(X)$ such that 
$c_j+d_j\le e$ and $p_a(\h{M}_j)\ge g$. Note that this is just a simple numerical game which is not too complicated, especially when  the value $e=\deg\h{E}=\deg \h{D}^\vee$ is low. Hence it trivially follows that 
 $$\h{G}'=\cup \h{G}^\vee\subset \underset{j}{\cup}~\h{G}_{\Sigma_{\h{M}_j,g}.}$$

\vni
(c)
For each $\h{M}_j$,  we  determine if  a general element of $\h{G}_{\Sigma_{\h{M}_j,g}}$
has a very ample residual series, 
 which implies   $$\h{G}_{\Sigma_{\h{M}_j,g}}\subset\h{G}'. $$
 
 \item[(iv)]
 If there is only one $\h{M}_j$ left such that $\h{G}_{\Sigma_{\h{M}_j,g}}\subset\h{G}'$, say $\h{M}_j=\h{M}_0$ we then may deduce that $ \h{G}_{\Sigma_{\h{M}_0,g}}=\h{G}'.$
 
 \item[(v)]
The irreducibility of $\widetilde{\h{G}}_\h{L}\stackrel{bir}{\cong}\h{G}'=\h{G}_{\Sigma_{\h{M}_0,g}}$  follows from the irreducibility of  ${\Sigma_{\h{M}_0,g}}$ or 
$\h{G}_{\Sigma_{\h{M}_0,g}}$.
\end{enumerate}

\end{defrmk}

\vni
\noindent
We deal with the case $5\le r\le 8$ using the above roadmap in the following theorem. The proofs of the irreducibility of $\HL{2r+4,r+8,r}$ for $r=3,4$ are of somewhat different flavor which shall be presented in the next section.
\begin{thm}\label{irreducible} For $5\le r\le 8$, $\HL{2r+4,r+8,r}$ is irreducible unless $r=7$.
\end{thm}
\begin{proof} Set  $\h{E}=g^3_{10}=\h{D}^\vee$ for a general $\h{D}\in\h{G}\subset\widetilde{\h{G}}_\h{L}\subset\h{G}^r_d$. We recall that $\h{E}$ is birationally very ample; cf. Lemma  \ref {easylemma} or the proof of the non-existence of $\HL{2r+4,r+8,r}$ for $r\ge 9$ in Theorem \ref{main} (c). Furthermore, $\h{E}$ is  base-point-free if $g=r+8\ge 13$; if $\h{E}$ has non-empty base locus, then $g\le \pi (9,3)=12$. We also note that the image curve $C_{\h{E}}\subset\PP^3$ induced by $\h{E}$ lies on a quadric surface; if $C_{\h{E}}$ does not lie on a quadric, then one has $$g\le p_a(C_{\h{E}})\le \pi_1(10,3)=\frac{(10-1)(10-2)}{6}=12.$$ 
Since curves on a quadric cone is a specialization of curves on a non-singular quadric surface by \cite[Chapt. 2]{Zeuthen} or \cite[Theorem, page 74]{Brevik},
we may assume that $C_\h{E}$ lies on a non-singular quadric surface.
We note that for 
every genera $g=13,14,15,16$ there exists an integral curve  of degree $10$ lying on a non-singular quadric surface $X\subset\PP^3$ with $\delta = 0,1,2,3$ nodes as its only singularities and hence $\Sigma_{\h{M},\delta}\neq\emptyset$ for some $\h{M}\in Pic(X)$
; cf. \cite[Corollary 4.6]{AC1}. 

\vni
\noindent The case $r=8$ was treated already in the proof of Theorem \ref{main}, and we assume $5\le r\le 7$. In the previous section, we demonstrated the existence of $\HL{2r+4,r+8,r}$ by showing that a general member in a certain very ample linear system $(a; b_1,\cdots , b_{t+1})$ on $S_{t+1}\rightarrow \PP^2$ is embedded into a Del-Pezzo surface in $\PP^r = \PP^{9-t-1}$ as a curve of degree $d=2r+4$ and genus $g=r+8$; say $a=8, b_1=b_2=3$ and  $b_i=2$ ($3\le i\le t+1$) if $t=2,3$. Here we show that these are the only possible family of curves if $r\neq 7$. Indeed, what we shall show is that $C_\h{E}\in\Sigma_{\h{M},\delta}$ for $\delta =t=2,3$ where $\h{M}=|\h{O}_{\PP^1\times\PP^1}(5,5)|$.

\vni
\noindent
Recall that a general element of any component of the equi-generic 
Severi variety $\Sigma_{\h{M},g}$ is a nodal curve by Remark \ref{Severi} (i), and hence we may assume that $C_\h{E}$ is a nodal curve with exactly $t$ nodes $q_1,\cdots ,q_t$.

\vni
\noindent Let $C_\h{E}\in |\h{O}_{\PP^1\times\PP^1}(c,d)|$. Since $\h{E}$ is base-point-free, $$c+d=10 \textrm{ and } g\le p_a(C_\h{E})=(c-1)(d-1),$$ there are just two possibilities for $(c,d)$ if $13\le g\le 15$;   $(c,d)=(5,5)$ or $(4,6)$. Hence we have 
 $$\h{G}'=\cup\h{G}^\vee\subset\h{G}_{\Sigma_{\h{M}.g}}\cup~ \h{G}_{\Sigma_{\h{N}.g}},$$
 where $\h{M}= |\h{O}_{\PP^1\times\PP^1}(5,5)|$ and $\h{N}= |\h{O}_{\PP^1\times\PP^1}(4,6)|$. 
 
 \vni
 \noindent
First assume that  $\h{M}= |\h{O}_{\PP^1\times\PP^1}(5,5)|$ i.e. $C_\h{E}\in \h{M}$,  $g=p_a(C_\h{E})-t=16-t$, where $t=1,2,3, 4$. As a matter of fact, we only need up to $t\le 3$ in this theorem. However, in order to handle the case $g=12$ and $r=4$ in the next section, we assume $t\le 4$ and the all the computations remain valid.  

\vni
\noindent
First we blow up the non-singular quadric $X$ at $q_1$ to get surface $S_2$ which is the surface blown up at two points in the projective plane; recall that a non-singular quadric in $\PP^3$ is a blowing up of $\PP^2$ at two points and then blown down along the proper transformation of the line through the two points. 

\noindent
If  $t\ge 2$, we first blow up at $q_1$ to get $S_2$ and then blow up  at the remaining points $\tilde{q}_2. \cdots , \tilde{q}_t$ -- which are the inverse images of $q_2, \cdots q_t$ under the first blow up $S_2\rightarrow X$ at $q_1$ -- arriving at the surface $S_{t+1}$. Now the proper transformation $\widetilde{C}_\h{E}\subset S_{t+1}$ of $C_\h{E}$ under these successive blow ups is smooth since $C_\h{E}$ is a nodal curve. We set
$$\widetilde{C}_\h{E}\in (a; b_1,\cdots b_{t+1}).$$  
Since the exceptional divisor of the first blow up at the node $q_1\in C_\h{E}\subset X$ is $l-e_1-e_2$, 
\begin{equation}
\widetilde{C}_\h{E}\cdot (l-e_1-e_2)=a-b_1-b_2=2.
\end{equation}
Since the proper transformations of the two rulings of $X\subset\PP^3$ -- which are of  the classes $l-e_i  ~(i=1,2)$ on $S_{t+1}$ -- cut out two distinct $g^1_5$'s on $\widetilde{C}_\h{E}$, we have
\begin{equation}
(l-e_i)\cdot (a;b_1,\cdots ,b_{t+1})=a-b_i=5,  ~i=1,2.
\end{equation}

\noindent
Let $C^1_\h{E}\subset S_2$ be the proper transformation of $C_\h{E}\subset\PP^3$ under the first blow up $S_2\rightarrow X$ at $q_1$; hence $C^1_\h{E}=\widetilde{C}_\h{E}$ if $t=1$.
For $t\ge2$, note that  $\tilde{q}_2, \cdots , \tilde{q}_t$  remain as  nodal singular points of $C^1_\h{E}$ and hence
$$e_{i+1}\cdot (a;b_1,\cdots ,b_{t+1})=b_{i+1}=2, \quad i=2,\cdots ,t$$ 
where $e_{i+1}$ ($i=2,\cdots ,t$) are the exceptional divisors over $\tilde{q}_i$ under the blow up $S_{t+1}\rightarrow S_2$.

\vni
\noindent
Solving equations (2) and (3),  we have $(a; b_1, b_2)=(8;3,3)$ and hence
$$\widetilde{C}_\h{E}\in(a;b_1, b_2,b_3, \cdots,b_{t+1})=(8;3,3,2,\cdots,2)=(8;3^2,2^{t-1}).$$
We check that $\h{E}^\vee$ is very ample. We have
\begin{align*}K_{S_{t+1}}+\widetilde{C}_{\h{E}}&=(-3;-1,-1,-1,\cdots,-1)+(8;3,3,2,\cdots,2)
\\&=(-3;-1^{t+1})+(8;3^2,2^{t-1})=(5;2^2, 1^{t-1}).
\end{align*}
Since $\h{E}$ is cut out by $|2l-e_1-e_2|=(2;1^2,0^{t-1})$ on $\widetilde{C}_\h{E}$,  $\h{E}^\vee$ is cut out on $\widetilde{C}_\h{E}$ by 
\begin{align*}|K_{S_{t+1}}&+\widetilde{C}_\h{E}-(2l-e_1-e_2)|=|(5;2^2, 1^{t-1})-(2l-e_1-e_2)|\\&=|(5;2^2, 1^{t-1})-(2;1^2,0^{t-1}|=
|(3;1^{t+1})|
\end{align*} which is very ample on $S_{t+1}$. Therefore, $\h{E}^\vee$ is very ample which gives an embedding  $$\widetilde{C}_\h{E}\hookrightarrow \PP^{9-(t+1)}=\PP^{8-t}=\PP^r$$ as a curve of degree 
\begin{align*}d&=(K_{S_{t+1}}+\widetilde{C}_\h{E}-(2l-e_1-e_2))\cdot\widetilde{C}_\h{E}= (3,1^{t+1})\cdot (8,3^2,2^{t-1})\\&=24-6-2\cdot(t-1)=2\cdot (8-t)+4=2r+4,
\end{align*}
 and therefore we have  $$\h{G}_{\Sigma_{\h{M},g}}\subset \h{G}'=\cup\h{G}^\vee.$$ 
\noindent
\vni
We now check if 
$\h{G}_{\Sigma_{\h{N},g}}\subset\h{G}'. $
Let $\h{E}\in  \h{G}_{\Sigma_{\h{N},g}}$, i.e. $C_\h{E}\in  |\h{O}_{\PP^1\times\PP^1}(4,6)|$. 

\noindent
(a)
Assume $g=15$ (and $r=7$).
In this case,  $C_\h{E}$ is smooth and  $\h{E}^\vee$ is cut out by 
$$|K_X+C_\h{E}-(1,1)|=|(-2,-2)+(4,6)-(1,1)|=|(1,3)|$$ which is very ample on $X$ inducing an embedding  
$$C_\h{E}\subset\PP^1\times\PP^1\stackrel{(1,3)}{\hookrightarrow} \PP^1\times\PP^3\subset\PP^7$$ as a curve of degree $$(4,6)\cdot (1,3)=12+6=18=2r+4.$$ 
Therefore we have $\h{G}_{\Sigma_{\h{N},g}}\subset\h{G}'$ and hence 
$$\h{G}'=\cup\h{G}^\vee =\h{G}_{\Sigma_{\h{N},g}}\cup\h{G}_{\Sigma_{\h{M},g}}, $$
showing that  
$$\widetilde{\h{G}}_\h{L}=(\h{G}_{\Sigma_{\h{N},g}})^\vee\cup(\h{G}_{\Sigma_{\h{M},g}})^\vee$$
is reducible with two distinct irreducible components of the same dimension; 
\begin{eqnarray*}
\dim\h{G}_{\Sigma_{\h{M},g}}&&=\dim\Sigma_{\h{M},g}-\dim\textrm{Aut}(X)=\dim\Sigma_{\h{M},\delta}-\dim\textrm{Aut}(X)\\&&=\dim |\h{O}_{\PP^1\times\PP^1}(5,5)|-1-\dim\textrm{Aut}(X)=28\\
\dim\h{G}_{\Sigma_{\h{N},g}}&&=\dim |\h{O}_{\PP^1\times\PP^1}(4,6)|-\dim\textrm{Aut}(X)=28.
\end{eqnarray*}
(b) For $r=5,6$ ($g=13,14$) we show the irreducibility of $\h{G}'$ i.e. $$\h{G}'=\h{G}_{\h{M},g}$$ by showing that  $$\h{G}_{\h{N},g}\nsubseteq\h{G}'$$
which is enough to show that the residual series of a  general element of $\h{G}_{\Sigma_{\h{N},g}}$ is not very ample. We carry out the same computation as we did for the case $\h{M}= |\h{O}_{\PP^1\times\PP^1}(5,5)|$. Since $C_\h{E}\in  |\h{O}_{\PP^1\times\PP^1}(4,6)|$, 
we set $g=p_a(C_\h{E})-s=15-s$, where $s=\delta=1,2,3$; here we only need up to $s\le 2$. However in order to handle the case $r=4$ in the next section, we make the computation even for $s=3$.
For $s=\delta =1$ and $g=14$, we blow up $X$ at the node $q_1\in C_\h{E}\subset X$ to get $S_2$. 
For $s\ge 2$, we first blow up $X$ at a node $q_1$ to get $S_2$ and then blow up  at the remaining nodes $\tilde{q}_2, \cdots , \tilde{q}_s$ -- which are the inverse images of remaining nodes $q_2, \cdots q_s$ of $C_\h{E}$ under the first blow up $S_2\rightarrow X$ at $q_1$ -- arriving at the surface $S_{s+1}$. The proper transformation $\widetilde{C}_\h{E}\subset S_{s+1}$ of $C_\h{E}$ under these successive blow ups is smooth and we set
$$\widetilde{C}_\h{E}\in (a; b_1,\cdots b_{s+1}).$$  
Since $q_1\in C_\h{E}\subset X$ is a node and the exceptional divisor over $q_1$ is $l-e_1-e_2$, 
\begin{equation}
\widetilde{C}_\h{E}\cdot (l-e_1-e_2)=a-b_1-b_2=2.
\end{equation}
Since the proper transformations of the two rulings of $X\subset\PP^3$ -- which are of the classes $l-e_i$ ($i=1,2$) on $S_{s+1}$ -- cut out two  base-point-free pencils  $g^1_4$ and  $g^1_6$ on $\widetilde{C}_\h{E}$, we have
\begin{equation}
(l-e_1)\cdot\widetilde{C}_\h{E}=a-b_1=4, ~~
(l-e_2)\cdot\widetilde{C}_\h{E}=a-b_2=6.
\end{equation}
Let $C^1_\h{E}\subset S_2$ be the proper transformation of $C_\h{E}\subset\PP^3$ under the first blow up $S_2\rightarrow X$ at $q_1$; hence $C^1_\h{E}=\widetilde{C}_\h{E}$ if $s=1$.
For $s\ge2$, note that  $\tilde{q}_2, \cdots , \tilde{q}_s$  remain as  nodal singular points of $C^1_\h{E}$ and hence
$$e_{i+1}\cdot (a;b_1,\cdots ,b_{s+1})=b_{i+1}=2, \quad i=2,\cdots ,s$$ 
where $e_{i+1}$ ($i=2,\cdots ,s$) are the exceptional divisors over $\tilde{q}_i$ under the blow up $S_{s+1}\rightarrow S_2$.

\noindent
Solving the equations (4) and (5), we  have $(a; b_1, b_2)=(8;4,2)$
and hence
$$\widetilde{C}_\h{E}\in(a;b_1, b_2,b_3, \cdots,b_{s+1})=(8;4,2,2^{s-1}).$$
We claim that $\h{E}^\vee$ is not very ample. We have
$$K_{S_{s+1}}+\widetilde{C}_{\h{E}}=(-3;-1^{s+1})+(8;4,2,2^{s-1})=(5;3,1^{s}).$$ 
Since $\h{E}=g^3_{10}$ is cut out by $|2l-e_1-e_2|=(2;1^2,0^{s-1})$ on $\widetilde{C}_\h{E}$, $\h{E}^\vee$ is cut out on $\widetilde{C}_\h{E}$ by 
\begin{align*}|K_{S_{s+1}}&+\widetilde{C}_\h{E}-(2l-e_1-e_2)|=|(5;3, 1^s)-(2l-e_1-e_2)|\\&=|(5;3, 1^s)-(2;1^2,0^{s-1})|=
|(3;2,0,1^{s-1})|
\end{align*}
which is not very ample on $S_{s+1}$; cf. \cite{sandra}. 
Moreover, $(3;2,0,1^{s-1})_{|\widetilde{C}_\h{E}}$ is not very ample; note that 
$$\widetilde{C}_\h{E}\cdot e_2=(8;4,2^s)\cdot e_2=2.$$
We also note that $(3;2,0,1^{s-1})\cdot e_2=0$ and hence $e_2$ is contracted by the morphism induced by the linear system $(3;2,0,1^{s-1})$ by which $\h{E}^\vee$ on $\widetilde{C}_\h{E}$ is cut out. Therefore it follows that
$$|K_{\widetilde{C}_\h{E}}-\h{E}|=(3;2,0,1^{s-1})_{|\widetilde{C}_\h{E}}$$ is not very ample.
This shows that  $$\h{G}_{\h{N},g}\nsubseteq\h{G}'$$
 and hence 
$$\h{G}'=
\h{G}_{\Sigma_{\h{M},g}},$$
from which the irreducibility of $\widetilde{\h{G}}_\h{L}$ and $\HL{2r+4,r+8,r}$ for $r=5,6$ follows. \end{proof}

\section{Irreducibility of $\HL{2r+4,r+8, r}$ for $r=3, 4$.}
In this section, we continue to prove the irreducibility of the peculiar Hilbert schemes $\HL{2r+4,r+8,r}$ for $r=3, 4$, whose proofs are lengthier and utilizing somewhat different methods from the cases $r\ge 5$. In the course of the proof of Theorem \ref{r=4},
we also prove  the reducibility of $\HL{g-2,g,3}$ for $g=12$.
\begin{thm} \label{r=4}$\HL{12,12,4}$ is irreducible.
\end{thm}
\begin{proof} Let $\h{E}=\h{D}^\vee$ for a general element  $\h{D}\in\h{G}\subset\widetilde{\h{G}}_\h{L}$. Recall that $\h{E}$ is birationally very ample (possibly with non-empty base locus) and $C_\h{E}\subset\PP^3$ denotes the image curve induced by $\h{E}$. Since $g=12=\pi_1(10,3)$, $C_\h{E}$ does not necessarily lie  on a quadric surface in $\PP^3$.  $C_\h{E}$ may lie either on a quadric or on an irreducible surface of higher degree.  We need to consider these two cases separately.  

\vni
(i) Assume that  $C_\h{E}\subset\PP^3$ lies on a quadric. 
Again, by Remark \ref{Severi} (i) and \cite{Zeuthen} we may assume that $C_\h{E}$ lies on a non-singular quadric $X\subset\PP^3$ and has $\delta$ nodes as its only singularities. The possibilities are

(a) $C_\h{E}\in  |\h{O}_{\PP^1\times\PP^1}(5,5)|$ and $\delta =4$,

(b) $C_\h{E}\in  |\h{O}_{\PP^1\times\PP^1}(4,6)|$ and $\delta =3$,

(c) $C_\h{E}\in  |\h{O}_{\PP^1\times\PP^1}(3,7)|$ and $C_\h{E}$ is non-singular, 

(d) $C_\h{E}\in  |\h{O}_{\PP^1\times\PP^1}(4,5)|$, $\h{E}$ has non-empty base locus and $C_\h{E}$ is non-singular. Hence we have
$$\h{G}'\subset \h{G}_{\Sigma_{(5,5),g}}\cup \h{G}_{\Sigma_{(4,6),g}}\cup \h{G}_{\Sigma_{(3,7),g}}\cup \h{G}_{\Sigma_{(4,5),g}}.$$

\vni
\noindent We check if the reverse inclusion holds.

\vni
\noindent 
(a)  The case $C_\h{E}\in\h{M}= |\h{O}_{\PP^1\times\PP^1}(5,5)|, \delta =4$ was treated in the proof of Theorem \ref{irreducible}; we concluded that  $\widetilde{C}_\h{E}\in (8;3^2,2^3)$ for $\h{E}\in\h{G}_{\Sigma_{\h{M},\delta}}$. We also showed that  $\h{G}_{\Sigma_{\h{M},g}}\subset\h{G}'$ (or equivalently $(\h{G}_{\Sigma_{\h{M},g}})^\vee\subset\widetilde{\h{G}}_\h{L}$)
by showing that a general $\h{E}\in\h{G}_{\Sigma_{\h{M},g}}$ has very ample residual series.

\vni
\noindent (b) $C_\h{E}\in\h{N}= |\h{O}_{\PP^1\times\PP^1}(4,6)|$ and $\delta =3$; we showed in the proof of Theorem \ref{irreducible} that $\h{E}^\vee$ is not very ample for  a  general element  $\h{E}\in\h{G}_{\Sigma_{\h{N},g}}$  and hence $\h{G}_{\Sigma_{\h{N},g}}\nsubseteq\h{G}'$.

\vni
\noindent
(c) $C_\h{E}\in  |\h{O}_{\PP^1\times\PP^1}(3,7)|$ and $C_\h{E}$ is non-singular; $\h{E}^\vee$ is cut out by the series 
$$|K_X+C_\h{E}-(1,1)|=|(-2,-2)+(3,7)+(1,1)|=|(0,4)|$$ on $C_\h{E}$, which is not very ample and hence 
$\h{G}_{\Sigma_{(3,7),g}}\nsubseteq\h{G}'$.
However, it is worthwhile to remark that non-singular curves in $|\h{O}_{\PP^1\times\PP^1}(3,7)|$ on  smooth quadrics in $\PP^3$ form an irreducible  family $\h{I}_2\subset \h{H}^\h{L}_{10,12,3}$ of the expected dimension 
$$\dim\h{G}_{\Sigma_{(3,7),g}}+\dim\PP H^0(\PP^3,\h{O}_{\PP^3}(2))=40.$$ 

\noindent
(d) $C_\h{E}\in  |\h{O}_{\PP^1\times\PP^1}(4,5)|$, $\h{E}$ has non-empty base locus and $C_\h{E}$ is non-singular:  Let $\h{E}'=|\h{E}-q|$ be the moving part $\h{E}$, where $q$ is the base point of $\h{E}$.
We argue that $\h{E}^\vee$ is not very ample as follows.
Note that the residual series of $\h{E}'$ is cut out on the non-singular curve $C_{\h{E}'}=C_{\h{E}}$  by the very ample linear system 
$$|K_X+C_{\h{E}'}-(1,1)|=|(-2,-2)+(4,5)-(1,1)|=|(1,2)|$$
which induces an embedding $X\stackrel{\iota}{\longrightarrow}\PP^1\times\PP^2\subset\PP^5$. Under this embedding $\iota$,
the non-singular curve $C_{\h{E}'}$ has the image $\iota (C_{\h{E}'})\subset\PP^5$ which is a curve of degree $2g-2-\deg\h{E}'=13$. Note that  $$|K_{C_\h{E'}}-\h{E}'|=|K_{C_\h{E'}}-\h{E}+q|,$$ 
and $\h{E}^\vee=|K_{C_{\h{E'}}}-\h{E}|$ gives rise to a projection of  $\iota (C_{\h{E}'})$ into $\PP^4$ with center at $\iota (q)\in \iota (C_{\h{E}'})\subset\PP^1\times\PP^2\subset\PP^5$. Let $L$ be one of the rulings of $X$ such that $q\in L$, say $L=(0,1)\subset X$. We note that the image  of the line $L$ under $\iota$ is still a line on the image scroll $\iota (X) \subset\PP^5$; $(1,2)\cdot (0,1)=1$. We also note that $\iota (L)$ is a $4$-secant line of the non-singular curve $\iota (C_{\h{E}'})$ containing $\iota (q)$ since  $(4,5)\cdot (0,1)=4$.
However,  a projection from the point $\iota (q)$ on the $4$-secant line $\iota (L)$ of  $\iota (C_{\h{E}'})$ produces a singularity and hence $\h{E}^\vee$ is not very ample. Therefore $\h{G}_{\Sigma_{\h{L},g}}\nsubseteq\h{G}'$ for $\h{L}= |\h{O}_{\PP^1\times\PP^1}(4,5)|$. 
\vni
(ii) We assume $C_{\h{E}}$ does not lie on a quadric. Since $g=12\le p_a(C_\h{E})\le \pi_1(10,3)=12$ (cf. \cite[Theorem 3.13]{H1}), we see that $C_\h{E}$ is non-singular. By the exact sequence
$$0\rightarrow H^0(\PP^3, \h{I}_{C_\h{E}}(3)) \rightarrow H^0(\PP^3,\h{O}_{\PP^3}(3))\rightarrow H^0(C_{\h{E}},\h{O}_{C_{\h{E}}}(3)) $$
one checks that $C_\h{E}$ lies on an (unique) irreducible cubic surface. 

\vni
\noindent
(ii-a) We assume that $C_\h{E}$ lies on a  smooth cubic surface $S=S_6\subset\PP^3$.
We compute the class 
of $C_\h{E}$ as follows. 
Setting $C_\h{E}\sim al-\sum^6_{i=1}b_ie_i$, we have $$\deg C_\h{E} =3a-\sum b_i=10, C_\h{E}^2=a^2-\sum b_i^2=2g-2-K_S\cdot C_\h{E}=22+\deg C_\h{E}=32.$$
By Schwartz's inequality,  one has $$(\sum b_i)^2\le 6(\sum b_i^2)$$ and substituting 
$\sum b_i=3a-10$, $\sum b_i^2=a^2-32$ we obtain $$3a^2-60a+292\le 0,$$ implying 
$9\le a\le 11$. Hence, we arrive at the following possibilities for 
the $7$-tuple $(a; b_1, \cdots , b_6)$;
\begin{equation}\label{cubic}
\textrm{ (i) } (9; 3^5,2) ~~~~\quad \textrm{ (ii) }  (10; 4^2,3^4) ~~~~ \quad\textrm{ (iii) } (11; 4^5,3)
\end{equation}

\noindent and by a simple numerical check, in all the three cases the curve $C_\h{E}$ is linearly equivalent on $S$ to $D+3H$ where $D$ is one of the $27$ lines on the cubic.

\noindent
For each of the above three types of $7$-tuples $(a; b_1, \cdots , b_6)$, we set $L=\mathcal{O}_S(a;b_1,\cdots ,b_6)$. By Riemann-Roch on $S$ we have
\begin{eqnarray*}
h^0(S,L)&=&h^1(S,L)-h^2(S,L)+\chi(S)+\frac{1}{2}(L^2-L\cdot\omega_S)\\
&=&h^1(S,L)-h^2(S,L)+1+\frac{1}{2}(C_\h{E}^2+\deg C_\h{E})\\
&=&h^1(S,L)-h^2(S,L)+1+\frac{1}{2}(32+10)
\end{eqnarray*}
By Serre duality  we have, $$h^1(S,L)=h^1(S, \omega_S\otimes L^{-1})=h^1(S,\mathcal{O}_S(-(a+3)l+\sum(b_i+1)e_i).$$
Since $E:=(a+3)l-\sum(b_i+1)e_i$ is (very) ample we have $$h^1(S,L)=h^1(S,\mathcal{O}_S(-(a+3)l+\sum(b_i+1)e_i))=0$$ by Kodaira's vanishing theorem. 
We further note the divisor $-E$ is not linearly equivalent to an effective divisor; if it were, one would have $-E\cdot l=-(a+3)\ge 0$ whereas $l^2=1\ge 0$, a contradiction. Hence it follows that  $$h^2(S,L)=h^0(S,\omega_S\otimes L^{-1})=h^0(S,\mathcal{O}_S(-E)=0$$ and we obtain $$h^0(S,L)=22.$$ Therefore the sublocus $\mathcal{I}_3$ of $\mathcal{H}_{10,12,3}$ consisting of  curves lying on a smooth cubic has dimension equal to the minimal possible dimension of a component of the Hilbert scheme $\mathcal{H}_{10,12,3}$, i.e.
$$\dim\h{I}_3=\dim H^0(\mathbb{P}^3, \mathcal{O}(3))-1+\dim |\mathcal{O}_S(a;b_1,\cdots ,b_6)|=40=4\cdot 10$$ 
We also note that $\h{I}_3$ is not in the closure $\overline{\mathcal{I}_2}$ or vice versa since $\dim\h{I}_2=\dim\h{I}_3$. 
So far we have shown that $\HL{10,12,3}$ has two irreducible families $\overline{\h{I}_2}$ and $\overline{\h{I}_3}$ of the 
same dimension $4\cdot 10$.

\vni
\noindent
(ii-b) We assume that $C_\h{E}$ lies on a singular cubic surface. We now argue that a family of non-singular curves lying on a singular cubic surface does not constitute a component of $\h{H}_{10,12,3}$. Note that every singular cubic surface $S\subset \PP^3$ is one of the following three types.

(1) $S$ is a normal cubic surface with some double points only.

(2) $S$ is a normal cubic cone.

(3) $S$ is not normal, which may possibly be a cone. 

\noindent
For the case (1), let $S$ be a normal cubic surface which is not a cone. By a work due to John Brevik \cite[Theorem 5.24]{Brevik}, every curve on $S$ is a specialization of curves on a smooth cubic surface. Therefore we are done for the case (1). 

\noindent
For the case (2), we let $C$ be a smooth curve of degree $d$ and genus $g$ on a normal cubic cone $S$. We recall by \cite[Proposition 2.12]{Gruson}  that 

(2-a) $g=1+d(d-3)/6-2/3$ if $C$ passes through the vertex of $S$ 

(2-b) $g=1+d(d-3)/6$, otherwise.

\noindent
Note that  $(d,g)=(10,12)$ satisfies only (a) and $C$ passes through the vertex $q$ of the normal cone $S$. Hence $C$ is a triple covering of an elliptic curve through the projection with center at the vertex $q$. The following naive dimension count shows that 
there is no component of $\HL{10,12,3}$ (or $\HL{12,12,4})$ whose general element corresponds to a triple covering of an elliptic curve. Assume the existence of such a component $\h{G}^\vee \subset\h{G}^3_{10}$. By Castelnuovo-Severi inequality, a triple covering $C\stackrel{\tau}{\rightarrow} E$ of an elliptic curve of genus $g=12$ is not trigonal, bielliptic or a smooth plane quintic. Hence by a theorem of Mumford \cite[Theorem 5.2, page 193]{ACGH}, one has 
$$\dim W^3_{10}(C)\le 10-2\cdot 3-2$$
and therefore $$\dim\h{G}^\vee\le\dim  W^3_{10}(C)+\dim\h{X}_{3,1}=2g\lneq \lambda (10,12,3)=\lambda(12,12,4),$$
leading to a contradiction. Here $\mathcal{X}_{n,h}$ denotes the locus in $\mathcal{M}_g$ corresponding to curves
which are n-fold coverings of smooth curves of genus $h$, which has pure dimension $$2g+(2n-3)(1-h)-2;$$
cf. \cite[Theorem 8.23, p.828]{ACGH2}. Therefore we are done for the case (2).

\noindent
For the case (3), we let $C=C_\h{E}$ be a smooth curve of degree $d$ and genus $g$ on a non-normal cubic surface $S$. Recall that if $S$ is a cone, then $S$ is a cone over a singular plane cubic, in which case $S$ is a projection of a cone $S'$ over a twisted cubic in a hyperplane in  $\PP^4$ from a point not on $S'$. Furthermore,  the minimal desingularisation $\tilde S$ of $S'$ is isomorphic to the ruled surface 
$$\mathbb{F}_3=\PP (\mathcal{O}_{\PP^1}\oplus \mathcal{O}_{\PP^1}(3)),$$ 
which is the blow-up of the cone $S'$ at the vertex.
If $S$ is not a cone, then $S$ is a projection of a rational normal scroll $$S''\cong\tilde S\cong\mathbb{F}_1=\PP (\mathcal{O}_{\PP^1}(1)\oplus \mathcal{O}_{\PP^1}(2))\subset \PP^4$$  from a point not on $S''$. 
In both cases, we have $\text{Pic}~ \tilde S=\mathbb{Z}h\oplus \mathbb{Z}f \cong\mathbb{Z}^{\oplus 2}$, where $f$ is the class of a fiber of $\tilde S \rightarrow \PP^1$ and $h=\pi^*(\mathcal{O}_S(1))$ with $\tilde S\stackrel{\pi}{ \rightarrow} S$. Note that $h^2=3, f^2=0$, $h\cdot f=1$ and $K_{\tilde S}\equiv -2h+f$. Denoting by $\tilde C\subset\tilde S$ the proper transformation of the curve $C\subset S$, we let $k:=(\tilde C \cdot f)_{\tilde S}$ be the intersection number of $\tilde C$ and the fiber $f$ on $\tilde S$. We have 
$$\tilde C \equiv kh+(d-3k)f=kh+(10-3k)f.$$
By adjunction formula, it follows that
$$g=12=\frac{(2\cdot 10-3k-2)(k-1)}{2},$$ which does have an integer solution and we are done with the case (3).

\vni
\noindent Since we have exhausted all the possibilities, the reducibility of $\HL{10,12,3}=\h{H}_{10,12,3}$ which has two components $\overline{\h{I}_2}$ and $\overline{\h{I}_3}$ of the 
same minimal dimension readily follows.

\vni
\noindent To conclude the irreducibility of $\HL{12,12,4}$, it remains to show that the residual series of a very ample $\h{E}=g^3_{10}$
-- which induces an embedding into a smooth cubic surface  $S=S_6\subset\PP^3$ -- is not very ample.

\vni
\noindent Recall that the residual series of a complete hyperplane series of a general element $C_{\h{E}}\in\h{I}_3$ is cut out 
by the linear system 
$$|K_S+C_\h{E}-H|=|2K_S+C_\h{E}|,$$
i.e. $|2K_S+C_\h{E}|_{|C_\h{E}}=|K_{C_\h{E}}-\h{E}|$. We recall our previous computation (\ref{cubic}).

\noindent
(i) $C_\h{E}=(9,3^5,2)$: Note that 
$|2K_S+C_\h{E}|=(3,1^5,0)$ which is not very ample on $S_6$.   We further note that $C_\h{E}\cdot e_6=2$ whereas 
$(2K_S+C_\h{E})\cdot e_6=0$. Hence the restriction of the morphism on $S_6$ induced by  $|2K_S+C_\h{E}|=(3,1^5,0)$  -- which is the blowing down of $S_6$ along $e_6$ -- produces a singularity on the image curve in $S_5\subset\PP^4$. Hence $|K_{C_{\h{E}}}-\h{E}|$ is not very ample. 

\noindent
(ii) $C_\h{E}=(10;4^2,3^4)$: Note that 
$$|2K_S+C_\h{E}|=(4;2^2,1^4),$$ and  $(4;2^2,1^4)\cdot (l-e_1-e_2)=0$ whereas $C_\h{E}\cdot (l-e_1-e_2)=2$. Hence
the restriction on $C_\h{E}$ of the morphism induced by $|2K_S+C_\h{E}|$ -- which contracts $(l-e_1-e_2)$  to a point -- produces a singularity.

\noindent
(iii) $C_\h{E}\in (11; 4^5,3)$: In this case, 
 $$|2K_S+C_\h{E}|=(5;2^5,1),$$
and $(5;2^5,1)\cdot (2l-e_1-e_2-e_3-e_4-e_5)=0$ whereas $C_\h{E}\cdot (2l-e_1-e_2-e_3-e_4-e_5)=2$. Hence $|K_{C_{\h{E}}}-\h{E}|$ is not very ample by the same reason.

\vni
\noindent  Finally, we may conclude that $$\h{G}'=\h{G}_{\Sigma_{(5,5),\delta =4}},$$
therefore it follows that 
$\widetilde{\h{G}}_\h{L}$ and $\HL{12,12,4}$ are irreducible.
\end{proof}

\noindent
We quote the reducibility of $\h{H}_{10,12,3}=\h{H}^\h{L}_{g-2,g,3}$ which we showed in the proof  
of the above theorem as a corollary.

\vni
\begin{cor} $\h{H}_{10,12,3}=\h{H}^\h{L}_{g-2,g,3}$ is reducible consisting of two components. 
\end{cor}

\noindent
We go down to the case $r=3$ and show the irreducibility of $\HL{10,11,3}$.
\begin{thm}\label{linkage} For $r=3$, $d=g+r-4$ and $g=r+8$, $$\h{H}_{d,g,3}=\h{H}^\h{L}_{d,g,3}=\HL{10,11,3}$$ is irreducible of the expected dimension. 
\end{thm}
\begin{proof} 
We first note that a smooth curve $C\subset \PP^3$ of degree $d=10$ and genus $g=11$ does not lie on a quadric surface just because 
there is no integer solution to the equations
 $$a+b=10=d, ~(a-1)(b-1)=11=g.$$
 We remark that the existence of $\h{H}_{10,11,3}$ is clear by the existence of a non-singular curve in the very ample linear system $(8;3^2,2^4)$ on a smooth cubic surface.
\noindent
\vni
Since $\deg C=10$, $C$ lies on at most one cubic surface $S$. Assume $S$ is a smooth cubic surface. We recall that 
smooth space curves of degree $d$ and genus $g$ on a smooth cubic surface form a finite union of locally closed irreducible family in $\h{H}_{d,g,3}$ of dimension $d+g+18$ if $d\ge 10$ by \cite[Proposition B.1]{Gruson}. Since $d+g+18< 4d$ for $(d,g)=(10,11)$, it follows that the family $\h{I}_3$ consisting of curves lying on a smooth cubic does not constitute a component of $\h{H}_{10,11,3}$.
If $C$ lies on a singular cubic surface, we may use the same argument as in the proof (ii-b) of Theorem \ref{r=4} to conclude that such  case does not occur.
\vni
Let $C$ corresponds to a general element of a component of $\h{H}_{10,11,3}$. By the above discussion we may assume that  $C$ does not lie on a quadric or a cubic.
Note that every smooth curve $C$ of genus $g=11$ of degree $d=10$ in $\PP^3$ lies on at least $$h^0(\PP^3, \h{O}_{\PP^3}(4))-h^0(C,\h{O}_C(4))=35-(40-11+1)=5$$ independent quartics. Hence $C$ is 
is residual to a curve $D\subset\PP^3$ of degree $e=6$ and genus $h=3$ in a complete intersection of two (irreducible) quartics
by the following well-known formula relating degrees $d, e$ and genus $g, h$ of two directly linked curves in $\PP^3$ in a complete intersection of surfaces of degrees $s$ and $t$;
\begin{equation}\label{dl}
2(g-h)=(s+t-4)(d-e).
\end{equation}
Consider the locus
$$\Sigma\subset\mathbb{G}(1,\PP(H^0(\PP^3,\h{O}_{\PP^3}(4))))=\mathbb{G}(1,34)$$
of pencils of quartic surfaces whose base locus consists of a curve $C$ of degree $d=10$ and genus $g=11$ and a sextic $D$ of genus $h=3$ which are directly linked via a complete intersection of quartics, together with two obvious maps 
\[
\begin{array}{ccc}
\hskip -48pt\mathbb{G}(1,34)\supset\Sigma&\stackrel{\pi_C}\dashrightarrow
~~\h{I}_4\subset\h{H}_{10,11,3}\\
\\
\dashdownarrow\vcenter{\rlap{$\scriptstyle{{\pi_D}}\,$}}
\\ 
\\
\h{H}_{6,3,3}
\end{array}
\]
where $\h{I}_4$ is the image of $\Sigma$ under $\pi_C$.
A sextic  $D\subset\PP^3$ of genus $h=3$ lies on at least $$h^0(\PP^3,\h{O}_{\PP^3}(4))-h^0(D,\h{O}_D(4))=35-(24-3+1)=13$$ independent quartics. 
Note that $C\in\h{H}_{10.11.3}$ is directly linked to $D\in\h{H}_{6,3,3}$ which in turn is directly linked to a twisted cubic $F$ via complete intersection of two cubics. From the basic  relation
$$\dim H^1(\PP^3,\h{I}_C(m))=\dim H^1(\PP^3,\h{I}_D(s+t-4-m))$$
where $C$ and $D$ are directly linked via complete intersection of surfaces of degrees $s$ and $t$, we have 
\begin{equation}\label{h1}h^1(\PP^3,\h{I}_C(4))=h^1(\PP^3,\h{I}_D(0))=h^1(\PP^3,\h{I}_F(3+3-4-0))=0
\end{equation}
and 
$$h^1(\PP^3,\h{I}_D(4))=h^1(\PP^3,\h{I}_F(3+3-4-4))=0.$$

\noindent
 Therefore $\pi_D$ is generically surjective with fibers open subsets of $\mathbb{G}(1,12)$.  Since $\h{H}_{6.3.3}$ is known to be irreducible (cf. \cite{E1} ), it follows that  $\Sigma$ is irreducible and $$\dim\Sigma=\dim\mathbb{G}(1,12)+\dim\h{H}_{6.3.3}=22+4\cdot 6=46.$$ On the other hand, since every smooth curve $C\in\h{I}_4$ of degree $d=10$ and $g=11$ lies on exactly $5$ independent quartics by (\ref{h1}), $\pi_C$ is generically surjective with fibers open in $\mathbb{G}(1,4)$. Finally it follows that $\overline{\h{I}_4}=\h{H}_{10.11.3}$ -- which is the image of (irreducible) $\Sigma$ under $\pi_C$ --  is irreducible of dimension
$$\dim\Sigma-\dim\mathbb{G}(1,4)=46-6=4\cdot10.$$
\end{proof}

\vni
\noindent The following table is a summary of our discussion in sections 3 and 4 regarding $\HL{2r+4,r+8,r}$. $C_\h{E}$ denotes possibly singular curve in $\PP^3$ determined by the residual series $\h{E}$ of the hyperplane series of  $[C]\in \h{H}^\h{L}_{2r+4,r+8,r}$.
Note that the irreducibility of $\h{H}^\h{L}_{2r+4,r+8,r}$ has been fully worked out  in Theorems \ref{irreducible}, \ref{r=4} \ref{linkage} and all curves described in the table form a dense open set in a component of $\h{H}^\h{L}_{g+r-4,g,r}$, except $r=3$ and $C\in (8;3^2, 2^4)$, which form an irreducible family of dimension $39< 40.$  

\newpage
\begin{table}[ht]
\caption{$\h{H}^\h{L}_{g+r-4,g,r}$=$\h{H}^\h{L}_{2r+4,r+8,r}\neq\emptyset$ for $3\le r\le 8$} 
\centering 
\begin{tabular}{c c c c} 
\hline\hline 
{\footnotesize $(d,g,r)$} & {\footnotesize Description of $C\in \h{H}^\h{L}_{g+r-4,g,r}$}  & {\footnotesize $C_\h{E}$} & {\footnotesize Irreducibility} \\ [0.5ex]
\hline 
\hline 
{\footnotesize$(10,11,3)$ } & {\footnotesize General element is directly linked to}& {\footnotesize }&{\footnotesize Yes} \\
{\footnotesize } & {\footnotesize $\widetilde{C}\in\HL{6,3,3}$ in a complete intersection}& {\footnotesize } \\
{\footnotesize } & {\footnotesize of two quartics.}& {\footnotesize}& {\footnotesize} \\
{\footnotesize } & {\footnotesize $C\in(8;3^2,2^4)$ on a Del Pezzo in $\PP^3$ }  & {\footnotesize  $C_\h{E}\in\Sigma_{|(5,5)|,5}$} & {\footnotesize} \\
{\footnotesize } & {\footnotesize which does {\bf not} constitute a component. }& {\footnotesize}& {\footnotesize} \\[0.5ex]
{\footnotesize $(12,12,4)$} & {\footnotesize $C\in(8;3^2,2^3)$ on a Del Pezzo in $\PP^4$ } & {\footnotesize $C_\h{E}\in\Sigma_{|(5,5)|,4}$} &{\footnotesize Yes}  \\ [0.5ex]
{\footnotesize $(14,13,5)$} & {\footnotesize $C\in (8;3^2,2^2)$ on a Del Pezzo in $\PP^5$ } & {\footnotesize$C_\h{E}\in\Sigma_{|(5,5)|,3}$} &{\footnotesize Yes}  \\[0.5ex]
{\footnotesize $(16,14,6)$} & {\footnotesize  $C\in (8;3^2,2)$ on Del-Pezzo in $\PP^6$}& {\footnotesize $C_\h{E}\in\Sigma_{|(5,5)|,2}$} &{\footnotesize Yes}  \\[0.5ex]
{\footnotesize  $(18,15,7)$} & {\footnotesize  $C\in (8;3^2)$ on Del-Pezzo in $\PP^7$}& {\footnotesize  $C_\h{E}\in\Sigma_{|(5,5)|,1}$
}&{\footnotesize {\bf No}} \\
{\footnotesize } & {\footnotesize or }& {\footnotesize } \\
{\footnotesize} & {\footnotesize $C\cong C_\h{E}\in |(4,6)|$ on $X\cong\PP^1\times\PP^1\subset\PP^3$}& {\footnotesize $C_\h{E}\in\Sigma_{|(4,6)|,0}$} &{\footnotesize}  \\ 
{\footnotesize } & {\footnotesize embedded into $\PP^7$ by  $|(1,3)|$ on $X$}& {\footnotesize} \\ [0.5ex]
{\footnotesize $(20,16,8)$} & {\footnotesize $C\cong C_\h{E}\in |(5,5)|$ on $X\cong\PP^1\times\PP^1\subset\PP^3$}& {\footnotesize $C_\h{E}\in\Sigma_{|(5,5)|,0}$} &{\footnotesize Yes} \\ 
{\footnotesize } & {\footnotesize embedded into $\PP^8$ by  $|(2,2)|$ on $X$}& {\footnotesize }& {\footnotesize} \\ [0.5ex]
{\footnotesize $(2r+4,r+8,r)$}& {\footnotesize $\HL{2r+4,r+8,r}=\emptyset$}  & {\footnotesize}  \\ 
{\footnotesize $r\ge 9$}& {\footnotesize}  & {\footnotesize}  \\
\hline 
\end{tabular}
\label{3r8} 
\end{table}

\section{Irreducibility and existence of $\HL{g+r-4,r+9,r}$ and beyond}

We have seen another peculiar Hilbert scheme of linearly normal curves $\HL{g+r-4,r+9,r}$ which is non-empty only for $3\le r\le 11$; Theorem \ref{main} (d).
The following table indeed provides the existence of the corresponding Hilbert scheme which has been postponed.  We stress that the curves described in the following table do not necessarily represent a general element in a component. As in the previous table, $C_\h{E}$ denotes possibly singular curve in $\PP^3$ determined by the residual series $\h{E}$ of the complete and very ample hyperplane series of $[C]\in \h{H}^\h{L}_{g+r-4,g,r}$ and $\Delta_{\h{E}}$ is the base locus of $\h{E}$.
We also stress that the irreducibility of $\HL{g+r-4,r+9,r}$ has not been determined in a couple of cases. In a remark after the table, we provide some  brief explanations for those $\HL{g+r-4,r+9,r}$ which we know of their irreducibility (or reducibility).
\newpage
\begin{table}[ht]
\caption{{\footnotesize$\h{H}^\h{L}_{g+r-4,g,r}$=$\h{H}^\h{L}_{2r+5,r+9,r}\neq\emptyset$ for $3\le r\le 11$}} 
\centering 
\begin{tabular}{c c c c c c} 
\hline
{\footnotesize $(d,g,r)$} & {\footnotesize Description of $C\in \h{H}^\h{L}_{g+r-4,g,r}$}  & {\footnotesize $C_\h{E}$} & {\footnotesize Irreducibility}\\
\hline 
{\footnotesize $(11,12,3)$} & {\footnotesize $C\in(8;3^2,2^3,1)$ on $S_6\hookrightarrow\PP^3$ }  & {\footnotesize  $C_\h{E}\in\Sigma_{|(5,5)|,4}$} &{\footnotesize Yes}\\
{\footnotesize } & {\footnotesize which do {\bf not} form a component }& {\footnotesize and $\Delta_{\h{E}}\neq\emptyset$} \\
{\footnotesize } & {\footnotesize or }& {\footnotesize} \\
{\footnotesize } & {\footnotesize General element is directly linked to}& {\footnotesize } \\
{\footnotesize } & {\footnotesize $\widetilde{C}\in\h{H}_{5,0,3}$ in a c.i. of two quartics}& {\footnotesize } \\[0.5ex]
{\footnotesize $(13,13,4)$} & {\footnotesize $C\in(8;3^2,2^2,1)$ on $S_5\hookrightarrow\PP^4$ }  & {\footnotesize  $C_\h{E}\in\Sigma_{|(5,5)|,3}$}&{\footnotesize Don't know yet} \\
{\footnotesize } & {\footnotesize  }& {\footnotesize  and $\Delta_{\h{E}}\neq\emptyset$} \\
{\footnotesize } & {\footnotesize or }& {\footnotesize} \\
{\footnotesize } & {\footnotesize $C\in (9;4,3,2^6)$ on $S_8\rightarrow\PP^2$}& {\footnotesize $C_\h{E}\in\Sigma_{|(5,6)|,7}$} \\
{\footnotesize } & {\footnotesize embedded into $\PP^4$ by  $(4;2,1^7)$}& {\footnotesize  and $\Delta_{\h{E}}=\emptyset$} \\[0.5ex]
{\footnotesize $(15,14,5)$} & {\footnotesize $C\in(8;3^2,2,1)$ on $S_4\hookrightarrow\PP^5$ }  & {\footnotesize  $C_\h{E}\in\Sigma_{|(5,5)|,2}$} &{\footnotesize Don't know yet.}\\
{\footnotesize } & {\footnotesize  }& {\footnotesize and $\Delta_{\h{E}}\neq\emptyset$} &{\footnotesize}\\
{\footnotesize } & {\footnotesize or }& {\footnotesize} \\
{\footnotesize } & {\footnotesize $C\in (9;4,3,2^5)$ on $S_7\rightarrow\PP^2$}& {\footnotesize $C_\h{E}\in\Sigma_{|(5,6)|,6}$} \\
{\footnotesize } & {\footnotesize embedded into $\PP^5$ by  $(4;2,1^6)$}& {\footnotesize and $\Delta_{\h{E}}=\emptyset$} \\ [0.5ex]
{\footnotesize $(17,15,6)$} & {\footnotesize $C\cong C_\h{E}\in(10;4,3^5)$ on $S_6\hookrightarrow\PP^3$ } & {\footnotesize $C_\h{E}$ is on cubic.} &{\footnotesize No} \\ 
{\footnotesize } & {\footnotesize embedded into $\PP^6$ by  $(4;2,1^5)$}& {\footnotesize and $\Delta_{\h{E}}=\emptyset$} \\
{\footnotesize  } & {\footnotesize or }& {\footnotesize } \\
{\footnotesize } & {\footnotesize  $C\in (9;4,3,2^4)$ on $S_6\rightarrow\PP^2$}& {\footnotesize $C_\h{E}\in\Sigma_{|(5,6)|,5}$} \\
{\footnotesize } & {\footnotesize embedded into $\PP^6$ by  $(4;2,1^5)$}& {\footnotesize and $\Delta_{\h{E}}=\emptyset$} \\[0.5ex]
{\footnotesize $(19,16,7)$} & {\footnotesize $C\in (9;4^2)$ on $S_2\hookrightarrow\PP^7$ } & {\footnotesize$C_\h{E}\in\Sigma_{|(5,5)|,0}$}  &{\footnotesize Don't know yet}\\
{\footnotesize } & {\footnotesize or }& {\footnotesize }&{\footnotesize } \\
{\footnotesize } & {\footnotesize $C\in (9;4,3,2^3)$ on $S_5\rightarrow\PP^2$}& {\footnotesize $C_\h{E}\in\Sigma_{|(5,6)|,4}$} \\
{\footnotesize } & {\footnotesize embedded into $\PP^7$ by $(4;2,1^4)$}& {\footnotesize and $\Delta_{\h{E}}=\emptyset$ } \\ [0.5ex]
 
{\footnotesize $(21,17,8)$} & {\footnotesize  $C\in (9;4,3,2^2)$ on $S_4\rightarrow\PP^2$}& {\footnotesize $C_\h{E}\in\Sigma_{|(5,6)|,3}$}&{\footnotesize Yes} \\
{\footnotesize } & {\footnotesize  embedded into $\PP^8$ by $(4;2,1^3)$} & {\footnotesize and $\Delta_{\h{E}}=\emptyset$}  \\[0.5ex]
{\footnotesize $(23,18,9)$} & {\footnotesize $C\cong C_\h{E}\in |(4,7)|$ on $X\cong\PP^1\times\PP^1\subset\PP^3$}& {\footnotesize $C_\h{E}\in\Sigma_{|(4,7)|,0}$} &{\footnotesize No}\\ 
{\footnotesize } & {\footnotesize embedded into $\PP^9$ by  $|(1,4)|$ on $X$}& {\footnotesize and $\Delta_{\h{E}}=\emptyset$} \\
{\footnotesize } & {\footnotesize or }& {\footnotesize } \\

{\footnotesize } & {\footnotesize  $C\in (9;4,3,2)$ on $S_3\rightarrow\PP^2$}& {\footnotesize  $C_\h{E}\in\Sigma_{|(5,6)|,2}$
} \\
{\footnotesize } & {\footnotesize embedded into $\PP^9$ by $(4;2,1^2)$}& {\footnotesize and $\Delta_{\h{E}}=\emptyset$}&{\footnotesize} \\ [0.5ex]
{\footnotesize $(25,19,10)$} & {\footnotesize $C\in (9;4,3)$ on $S_2\rightarrow\PP^2$}  & {\footnotesize $C_\h{E}\in\Sigma_{|(5,6)|,1}$}
&{\footnotesize Yes}  \\
{\footnotesize } & {\footnotesize embedded into $\PP^{10}$ by  $(4;2,1)$}  & {\footnotesize and $\Delta_{\h{E}}=\emptyset$}  \\[0.5ex]
{\footnotesize $(27,20,11)$}& {\footnotesize $C\cong C_\h{E}\in|(5,6)|$ on  $X\cong\PP^1\times\PP^1\subset\PP^3$}  & {\footnotesize  $C_\h{E}\in\Sigma_{|(5,6)|,0}$}&{\footnotesize Yes}  \\ 
{\footnotesize}& {\footnotesize embedded into $\PP^{11}$ by $|(2,3)|$ on $X$}  & {\footnotesize and $\Delta_{\h{E}}=\emptyset$}  \\
\hline 
\end{tabular}
\label{3r11} 
\end{table}

\begin{rmk} 
\begin{enumerate} 
\item[(i)] $(d,g,r)=(11,12,3)$: The proof of Theorem \ref{linkage} using elementary linkage theory works verbatim. $\h{H}_{11,12,3}=\HL{11,12,3}$ is irreducible and has the expected dimension, its general element corresponds to a curve which is directly linked to a rational quintic. 
\item[(ii)] $(d,g,r)=(17,15,6)$:  By noting that $g=\pi_1(11,3)=15$ and analyzing $C_{\h{E}}\subset\PP^3$ in a way similar to  part (ii) of the proof of Theorem \ref{r=4}, we arrive at the two possibilities; (a) $C_\h{E}\in \Sigma_{|(5,6)|,5}$ or (b) $C_\h{E}\in (10; 4,3^5)$
on $S_6\subset\PP^3$. Two families form different components by semi-continuity; note that $h^0(\PP^3,\h{I}_{C_2}(2))>h^0(\PP^3,\h{I}_{C_3}(2))$, $C_2\in\Sigma_{|(5,6)|,5}$, $C_3\in(10; 4,3^5)$ whereas $$\dim\h{G}_{(10; 4,3^5)}=29<\dim\h{G}_{\Sigma_{|(5,6)|,5}}=30.$$
\item[(iii)] $(d,g,r)=(19,16,7)$: We note that $g=r+9\gneq \pi_1(11,3)$ if $r\ge 7$, hence $C_\h{E}\subset\PP^3$ lies on a quadric. Therefore we have  $C_\h{E}\in \Sigma_{|(5,6)|,4}\cup\Sigma_{|(4,7)|,2}$ if $\h{E}$ is base-point-free. One may also show that 
$\h{G}_{\Sigma_{|(4,7)|,2}}\nsubseteq\h{G}'$ by arguing that $\h{E}^\vee$ is not very ample for $C_\h{E}\in (4,7)$. 
However, if $\h{E}=|g^3_{10}+q|=|\h{E}'+q|$ has a base point $q$, $$|K_{C_\h{E}}-\h{E}|=|K_{C_{\h{E}'}}-\h{E}|=|K_{C_{\h{E}'}}-{\h{E}'}-q|,$$ which is still very ample for a general $q$.  The family of curves consisting of (smooth) curves in $|\h{O}_{\PP^1\times\PP^1}(5,5)|$  which is embedded into $\PP^8$  induced by  $|K_{C_{\h{E}'}}-{\h{E}'}|=(2,2)_{|{{C_\h{E}}={C_{\h{E}'}}}}$ and then projected from a general point on the image curve in $\PP^8$$$C_\h{E}=C_{\h{E}'}\stackrel{(2,2)}{\hookrightarrow }\PP^2\times\PP^2\subset\PP^8\stackrel{\pi_q}{\longrightarrow}\PP^7$$ has dimension $30\ge\lambda (d,g,r)=29$.
However, it is not totally clear if this family is in the boundary of the family of curves arising from $\h{G}_{\Sigma_{|(5,6)|,4}}^\vee$, which has dimension $31$.
\item[(iv)] $(d,g,r)=(21,17,8)$: We have $C_\h{E}\in\Sigma_{|(5,6)|,3}\cup\Sigma_{|(4,7)|,1}$. However it turns out that $\h{G}_{\Sigma_{|(4,7)|,1}}\nsubseteq\h{G}'$. Hence $\h{G}'=\h{G}_{\Sigma_{|(5,6)|,3}}$ is irreducible.
The irreducibility of 
$\HL{21,17,8}$ follows from the irreducibility of $\Sigma_{|(5,6)|,3}$.

\item[(v)] $(d,g,r)=(23,18,9)$: One can check $C_\h{E}\in\Sigma_{|(5,6)|,2}\cup\Sigma_{|(4,7)|,0}$ and $\h{G}'=\h{G}_{\Sigma_{|(5,6)|,2}}\cup\h{G}_{\Sigma_{|(4,7)|,0}}$. Hence $\HL{23, 18, 9}$ is reducible with two components.
\item[(vi)] $(d,g,r)=(25,19,10)$: $C_\h{E}\in\Sigma_{|(5,6)|,1}$ is the only possibility. i.e. $\h{G}'=\h{G}_{\Sigma_{|(5,6)|,1}}$ and hence 
$\HL{25,19,10}$ is irreducible.
\item[(vii)] $(d,g,r)=(27,20,11)$: $C_\h{E}\subset\PP^3$ is an extremal curve in the linear system $|\h{O}_{\PP^1\times\PP^1}(5,6)|$ on a smooth quadric and the irreducibility of 
$\HL{27,20,11}$ follows from the irreducibility of $\HL{11,20,3}$.
\end{enumerate}
\end{rmk}

\vni
The following statement regarding the Hilbert scheme of linearly normal curves with index of speciality $\alpha =5$ is rather crude compared with Theorem \ref{main}. We omit 
the proof which is more involved but is similar to the proof of Theorem \ref{main}.  We assume $r\ge 6$ in the following statement for the sake of brevity, i.e. in order to avoid  subtle complexities in lower dimensions $3\le r\le 5$, which need to be considered separately. A proof of the following proposition in a more general formulation shall appear in an article under preparation; \cite{index5}.

\begin{prop} \label{sub}
\begin{enumerate}
\item[(a)] $\HL{g+r-5,g,r}=\emptyset$ for $g\le r+8$, $r\ge 6$. 
\item[(b)] $\HL{g+r-5,g,r}\neq\emptyset$
for  $g=r+9$ or for any $g\ge r+13$, $r\ge 6$.
\item[(c)] $\HL{g+r-5,g,r}=\emptyset$ for $g=r+10$, $r\ge 9$.
\item[(d)] $\HL{g+r-5,g,r}=\emptyset$ for $g=r+11$, $r\ge 12$. 
\item[(e)] $\HL{g+r-5,g, r}\neq\emptyset$ for $g=r+12$, $r\ge 13$. 
\end{enumerate}
\end{prop}

\begin{rmk} \begin{enumerate}
\item[(i)] In three cases 
(1)  $6\le r\le 8$ and $g=r+10$; 
(2)  $8\le r\le 11$ and $g=r+11$;
(3) $g=r+12$ and $10\le r\le 14$,
the curve $C_\h{E}$ induced by  birationally very ample
$\h{E}=g^4_e=\h{D}^\vee$ for general $\h{D}\in\h{G}\subset\widetilde{\h{G}}_\h{L}$ lies on a rational normal scroll $S\subset\PP^4$; $\pi_1(e,4)\lneq g\le\pi (e,4)$.
Therefore, as we did in previous sections, we expect that the existence and the irreducibility (or reducibility in certain cases) would 
follow by looking at (possibly singular) curves on a scroll and studying the Severi variety of nodal curves on the scroll $S$. 
\item[(ii)] For $g=r+12$, $\h{E}=\h{D}^\vee =g^4_{15}$ may induce a triple covering of an elliptic curve $C\stackrel{\phi}{\rightarrow} E$, i,e, $\h{E}=\phi^*(g^4_5)$.  One may show that  $\h{E}^\vee$ is very ample if $r\ge 13$, which implies $\HL{g+r-5,g, r}\neq\emptyset$ for $g=r+12$, $r\ge 13$.
\end{enumerate}
\end{rmk}
\bibliographystyle{spmpsci} 

\end{document}